\documentclass[12pt]{article}

\usepackage{amsmath,amsthm,amssymb,amsfonts}
\usepackage{tikz-cd} 
\usepackage{braket}
\usepackage{bbm}
\usepackage{enumitem}

\usepackage{graphics}
\usepackage{enumerate}
\usepackage{array}
\usepackage[english]{babel}
\usepackage{verbatim}
\usepackage{hyperref}
\usepackage{caption}

\usepackage{pdflscape}
\usepackage{rotating}
\usepackage{graphicx}

\usepackage{geometry} % to change the page dimensions
\usepackage{mathrsfs}
\usepackage{mathtext}
\usepackage{mathtools}

\usepackage[all,knot]{xy}

\newcommand{\bu}{\setlength{\unitlength}{1pt}\begin{picture}(2.5,2)
               (1,1)\put(2,2.5){\circle*{2}}\end{picture}}

\newcommand{\N}{\mathbb{N}}

\newcommand{\ot}{\otimes}
\newcommand{\op}{\oplus}
\newcommand{\mrm}{\mathrm}
\newcommand{\ox}{\overline{x}}
\newcommand{\oxa}{\overline{x_{1}}}
\newcommand{\oxb}{\overline{x_{2}}}
\newcommand{\K}{\mathbbm{k}}

\numberwithin{equation}{section}

\newtheorem{defi}[equation]{Definition}

\newtheorem{theo}[equation]{Theorem}

\newtheorem{lemma}[equation]{Lemma}

\title{Resolutions for Truncated \\ Ore Extensions}
\author{Dustin McPhate\\ }

\begin{document}

\maketitle
\thispagestyle{empty}
%\clearpage

\begin{abstract}
 We extend recent results in order to construct projective resolutions for modules over twisted tensor products of truncated polynomial rings. We begin by taking note of the conditions necessary to think of these algebras as a type of Ore extension. We then use this parallel with Ore extensions to develop a method for constructing  projective resolutions.  Finally we use the new construction to compute a resolution for a family of examples.
\end{abstract}

\section{Introduction}\label{sec:introduction}
Wanting to generalize the Eilenberg-Zilber Theorem to fiber spaces, Edgar Brown published a paper in 1959 on the study of the singular cohomology of fiber spaces arising in algebraic topology. In the process of doing so he  introduced what he called a twisted tensor product of algebras. The definition arose naturally out of his attempts to give an algebraic description of certain fibrations \cite{BRN}. His construction focused on tensor products of differential graded augmented algebras, or DGA algebras, where the twisting maps were induced by the differentiation maps. 

 In 1995, motivated by a question from non-commutative differential geometry, \v Cap, Schichl, and Van\v zura revisited the idea of a twisted tensor product of algebras. Given two algebras that describe two spaces, they wanted to know what would be an appropriate notion of the product of those spaces. Intuition from the non-commutative case allowed them to introduce a more general  and much more useful definition for a twisted tensor product  of unital algebras. This definition gave a new way of thinking about many common non-commutative algebras. In particular any algebra which is isomorphic as a vector space to the tensor product of two of its subalgebras under the canonical inclusion maps is also isomorphic to some twisted tensor product of those subalgebras \cite{CSV}. In the same paper they also gave the conditions needed for the multiplication induced by a twisted tensor product to be associative. 

For quite some time the majority of the study of the homology theory of twisted tensor products focused on calculating the co/homology of some particular examples. However in 2008 Bergh and Opperman obtained very strong results concerning the cohomology of a large class of twisted tensor products. They were interested in the cohomology groups over a quantum complete intersection and so looked at twisted tensor products  of graded algebras whose twisting maps arise from a bicharacter on the grading groups.  In \cite{BOP}  they showed that the Ext-algebra of this family of twisted tensor products can be constructed by taking a twisted tensor product of the Ext-algebras of the factors. Later Shepler and Witherspoon were looking to study deformations of twisted tensor product algebras and in order to do so they wished to be able to describe the homology theory of such algebras in terms of the homology theory of their factors. So in 2019 they published a paper giving the conditions necessary for resolutions  of modules of the factor algebras to be compatible with twisting maps \cite{RTTP}. They then, in the same paper, showed how to use these compatible resolutions to construct resolutions for the twisted tensor product of the factor algebras.

Included in \cite{RTTP} are some homological methods for a class of twisted tensor products called Ore extensions. In 1933, $\O$ystien Ore introduced a new class of noncommutative rings by generalizing earlier work by Hilbert and Schlessinger \cite{ORE}. These rings and their algebra counterparts came to be known as Ore extensions. The noncommutative multiplication in these algebras arises from the use of an automorphism and a derivation. By 1966 Gopalakrishnan and Sridharan were studying the homological properties of Ore extensions \cite{GOSR}. They were able to construct resolutions for certain classes of Ore extensions. In the mentioned paper of Shepler and Witherspoon is a method for constructing projective resolutions for any Ore extension.

In this paper we give a definition for a class of associative algebras which share many similarities with Ore extensions. We call them truncated Ore extensions and in fact one may think of these algebras as quotients of Ore extensions. Some examples include $U_{q}(\mathfrak{sl}_{2})^{+}$ the positive part of the quantized universal enveloping algebra of $\mathfrak{sl}_{2}$, the family of quantum algebras $A_{q}(0|2) \cong \K[x,y]/(xy - qyx, x^{2}, y^{2})$,  and the family of Nichols algebras $R \cong \K \langle x,y \rangle /(x^{p}, y^{p}, yx - xy - \frac{1}{2}x^{2})$ used in \cite{NWW}. We then use this parallel with Ore extensions to adapt the methods of \cite{RTTP} in order to construct projective resolutions for truncated Ore extensions. The projective resolution our construction gives for the Nichols algebra $R \cong \K \langle x,y \rangle /(x^{p}, y^{p}, yx - xy - \frac{1}{2}x^{2})$  is the same as the one constructed in \cite{NWW} and in the last portion of this paper we construct a resolution for a family of algebras which include $R$. 
\section{Preliminary Information}\label{sec:prelim}
Throughout this paper we assume $\K$ is a field and $n$ is a positive integer, $n \geq 2$. We use the common notation $\ox = x + (x^n) \in \K[x]/(x^n)$ and $\ot = \ot_{\K}$.  Let  $A, B$ be associative $\K$-algebras with multiplication maps $m_{A}$ and $m_{B}$.

\begin{defi}{\em
Let $\tau$ be a bijective $\mathbbm{k}$-linear map, $ \tau: B \ot A \rightarrow A \ot B$, for which $\tau(1_{B} \ot a) = a \ot 1_{B}$ , $\tau(b \ot 1_{A}) = 1_{A} \ot b$ for all $a \in A, \,\,\, b \in B$ and for which the compositions
\begin{equation}\tau (m_{B} \ot m_{A}) = (m_{A} \ot m_{B})  (1 \ot \tau \ot 1)  (\tau \ot \tau)(1 \ot \tau \ot 1)\end{equation}
as maps from $B \ot B \ot A \ot A$ to $A \ot B$. Then $\tau$ is called a $\textit{twisting map}$.}
\end{defi}

\begin{defi}{\em
Let $\tau$ be a twisting map. The $ \textit{twisted tensor product algebra},$ $ \, \, \, \, A\ot_{\tau}B, $ is the vector space $A \ot B$ with multiplication given by the map 
\begin{equation}
(m_{A} \ot m_{B}) (1 \ot \tau \ot 1)
\end{equation}
on $A \ot  B \ot A \ot  B$.}
\end{defi}
It is shown in \cite{CSV} that multiplication given by a twisting map is associative as a consequence of relation (2.2). We also note that since $\tau$ is bijective, $\tau^{-1}$ exists and there is a natural $\K$-algebra isomorphism $A \ot_{\tau} B \cong B \ot_{\tau^{-1}} A$.

Ore extensions are a specific class of twisted tensor products constructed in the following way. Let $A$ be any associative algebra. Let $\sigma$ be a $\K$-linear automorphism of $A$, that is  $\sigma \in \textrm{Aut}_{\K}(A)$.  Finally let  $\delta$ be a left $\sigma$-derivation of $A$, i.e. $\delta:A \rightarrow A$ such that 
$$ \delta(aa') = \delta(a)a' + \sigma(a)\delta(a') \, \, \, \, \, \, \textrm{for} \, \, \,  \textrm{all} \, \, \, \,  a, a' \in A.$$
\begin{defi} {\em
The $\textit{Ore extension} \, \,  A[x; \sigma, \delta]$ is the associative algebra with underlying vector space $A[x]$ and multiplication determined by that of $A$ and $\K[x]$ with the additional Ore relation
$$xa= \sigma(a)x + \delta(a) \, \, \, \, \, \, \textrm{for} \, \, \,  \textrm{all} \, \, \, \,  a \in A .$$}
\end{defi}
Thus if we let $B = \K[x]$ and $\tau$ be the twisting map induced by $\tau(x \ot a) = \sigma(a) \ot x + \delta(a) \ot 1$ then $A[x; \sigma, \delta] \cong A \ot_{\tau} B$. 

In this paper we wish to take the idea of an Ore extension and modify it slightly to cover a family of twisted tensor products who share a similar algebraic structure with Ore extensions. We thus define an algebra whose multiplication is determined similarly to an Ore extension but has $A[x]/(x^{n})$ as an underlying vector space for some integer $n$ instead of simply $A[x]$.
 
 We note that when using a quotient as our underlying vector space in order for a map $\tau$ generated by the Ore relation above to be a twisting map we must impose conditions on $\sigma$ and $\delta$. In order for $\tau$ to induce a well-defined map on the quotient, $(x^{n}) \ot A$ must be in $\ker(\tau)$. We will first define our new class of algebras and then afterward in Theorem 2.7 we derive the conditions on $\sigma$ and $\delta$ necessary for $\tau$ to induce a well defined associative multiplication.

\begin{defi}{\em
A $\textit{truncated Ore extension}$  $A[\overline{x}; \sigma, \delta]$ is an associative algebra with underlying vector space $A[x]/(x^{n})$ and multiplication determined by that of $A$ and of $\K[x]/(x^{n})$ with the additional Ore relation 
$$ \overline{x}a = \sigma(a)\overline{x} + \delta(a) \, \, \, \, \, \, \mathrm{for} \, \, \mathrm{ all} \, \, a \in A$$ 
for some $\sigma \in \mrm{Aut}_{\K}(A)$ and $\delta$ a left $\sigma$-derivation of $A$.}
\end{defi}

In a similar fashion as above we see that if $B = \K[x]/(x^{n})$ and $\tau$ is the twisting map induced by $\tau(\ox \ot a) = \sigma(a) \ot \ox + \delta(a) \ot 1$ for any $a \in A$, then the twisted tensor product of $A$ and $B$ under $\tau$, $A \ot_{\tau} B$, is isomorphic to the truncated Ore extension $A[\overline{x}; \sigma, \delta]$.

Now before we present the conditions on $\sigma$ and $\delta$ we mentioned earlier we must first introduce some notation in order to succinctly express these relations. Let $s_{(i_{1}, i_{2}, ...,i_{k})}(x_{1}, x_{2}, ...,x_{k})$ be the polynomial in $k$ noncommuting variables, $x_{1}, x_{2}, ...,x_{k}$, that is a sum of all possible products of $i_{1}$ copies of $x_{1}$,  $i_{2}$ copies of $x_{2}$, ..., and $i_{k}$ copies of $x_{k}$. For example 
$$s_{(2,2)}(x_{1},x_{2})= x_{1}^{2} x_{2}^{2} + x_{1}x_{2}^{2}x_{1} + x_{1}x_{2}x_{1}x_{2} + x_{2}x_{1}x_{2}x_{1} + x_{2}x_{1}^{2}x_{2} +  x_{2}^{2} x_{1}^{2}.$$ 
Thus through a slight abuse of our newly introduced notation we interpret $s_{(1,2)}(\sigma, \delta)$ to be the map 
$$s_{(1,2)}(\sigma, \delta) = \sigma \delta^{2} + \delta \sigma \delta +\delta^{2}\sigma$$ 
 where the product is defined to be composition of maps. 

\begin{theo}
Let $A$ be an associative algebra and $A[x; \sigma, \delta]$ be an Ore extension. Let $\tau$ be the twisting map associated with $A[x; \sigma, \delta]$.  If the maps $\sigma, \delta: A\rightarrow A$  satisfy the relations
\begin{equation}
s_{(i,j)}(\sigma, \delta) = 0 
\end{equation}
for all $i = 0, 1, ..., n-1$, $j = 1, 2, ...., n$ such that $i +j = n$, then $\tau$ induces a well defined multiplication on $A[\ox; \sigma, \delta]= A \ot_{\tau} \K[x]/(x^{n}).$ 

\end{theo}

\begin{proof}
Let $A$ be any associative algebra, $\widehat{B} = \K[x]$, $B = \K[x]/(x^{n})$, and $\tau$ be a twisting map from $\widehat{B} \ot A$ to $A \ot \widehat{B}$ given by $\tau(x \ot a) = \sigma(a) \ot x + \delta(a) \ot 1$ . Suppose $b_{0}, b_{1} \in \widehat{B}$ such that $b_{0} + (x^{n} ) = b_{1} + (x^{n}) \in B$. Since multiplication in $A[\ox; \sigma, \delta]$ is given by $(2.4)$ then for $a, a' \in A$ and $b \in \widehat{B}$, $(a \ot b_{0})(a' \ot b) = (a \ot b_{1}) (a' \ot b)$ would follow from $\tau(b_{0} \ot a') = \tau(b_{1} \ot a')$. And $\tau(b_{0} \ot a') = \tau(b_{1} \ot a')$ if and only if $\tau((b_{0} - b_{1}) \ot a' ) \in A \ot (x^{n})$. Thus in order for $\tau$  to induce a well defined twisting map from $B \ot A$ to $A \ot B$ we must have $(x^{n}) \ot A \subset \mathrm{ker}(\tau)$.  Since such a $\tau$ is a $\K$-linear twisting map it is  sufficient to show that $\tau(x^{n} \ot a) = 0$ for all $a \in A$. 

We will now show by induction that 
\begin{align}
\tau(x^{n} \ot a) = \sum_{i + j = n} s_{(i,j)}(\sigma, \delta)(a) \ot x^{i}.
\end{align}
By definition
\begin{align*}
\tau(x \ot a ) = \sigma(a) \ot x + \delta(a)  \ot 1 = s_{(1,0)}(\sigma, \delta)(a) \ot x + s_{(0,1)}(\sigma, \delta)(a) \ot 1.
\end{align*}
Now assume for $k \leq  l-1$ that 
\begin{align*}
\tau(x^{k} \ot a) = \sum_{i + j = k} s_{(i,j)}(\sigma, \delta)(a) \ot x^{i}.
\end{align*}
Then 
\begin{align*}
\tau(x^{l} \ot a) &= \tau(m_{B} \ot m_{A})( x \ot x^{l-1} \ot a \ot 1) \\
=(m_{A} \ot m_{B})(1 \ot \tau \ot 1)&(\tau \ot \tau)(1 \ot \tau \ot 1)(x \ot x^{l-1} \ot a \ot 1)\\
=(m_{A} \ot m_{B})(1 \ot \tau \ot 1)&(\tau \ot \tau)(x \ot ( \sum_{i + j = l-1} s_{(i,j)}(\sigma, \delta)(a) \ot x^{i}) \ot 1) \\
= (m_{A} \ot m_{B})(1 \ot \tau \ot 1)&(\tau \ot \tau)( \sum_{i + j = l-1} x \ot s_{(i,j)}(\sigma, \delta)(a) \ot x^{i} \ot 1)\\
= (m_{A} \ot m_{B})(1 \ot \tau \ot 1)&(\sum_{i + j = l-1}(\sigma(s_{(i,j)}(\sigma, \delta)(a)) \ot x + \delta(s_{(i,j)}(\sigma, \delta)(a)) \ot 1) \ot 1 \ot x^{i})\\
=  (m_{A} \ot m_{B})(1 \ot \tau \ot 1)&( \sum_{i + j = l-1}\sigma(s_{(i,j)}(\sigma, \delta)(a) \ot x \ot 1 \ot x^{i} \\
 &+ \sum_{i + j = l-1}\delta(s_{(i,j)}(\sigma, \delta)(a)) \ot 1 \ot 1 \ot x^{i}) \\
= (m_{A} \ot m_{B})(\sum_{i + j = l-1}\sigma(&s_{(i,j)}(\sigma, \delta)(a) \ot 1 \ot x \ot x^{i}  + \sum_{i + j = l-1}\delta(s_{(i,j)}(\sigma, \delta)(a)) \ot 1 \ot 1 \ot x^{i}) 
\end{align*}
\begin{align*}
&= \sum_{i + j = l-1}\sigma(s_{(i,j)}(\sigma, \delta)(a) \ot  x^{i+1}  + \sum_{i + j = l-1}\delta(s_{(i,j)}(\sigma, \delta)(a)) \ot  x^{i}\\
&= \sum_{i + j = l} (\sigma(s_{(i-1,j)}(\sigma, \delta)) + \delta(s_{(i,j-1)}(\sigma, \delta)))(a) \ot x^{i}  
\end{align*}
where we interpret $s_{(-1,n)}(\sigma, \delta) = s_{(n,-1)}(\sigma, \delta) = 0$.
Now $s_{(i,j)}(\sigma, \delta)$ is an expression which has as terms all possible arrangements of $i$ $\sigma$'s and $j$ $\delta$'s. We can group the terms into two sets, one which has all the terms which start with $\sigma$ and one which has all the terms which start with $\delta$. Since $s_{(i,j)}(\sigma, \delta)$ covers all possible arrangements then the terms that start with $\sigma$ contain all possible arrangements of $i-1$ $\sigma$'s and $j$ $\delta$'s. Similarly the terms which start with $\delta$ contain all possible arrangements of $i$ $\sigma$'s and $j-1$ $\delta$'s. Thus we may rewrite the expression $s_{(i,j)}(\sigma, \delta)$ in terms of this grouping to see that $$ s_{(i,j)}(\sigma, \delta) = \sigma( s_{(i-1,j)}(\sigma, \delta)) + \delta(s_{(i,j-1)}(\sigma, \delta))$$
with $s_{(0,j)}(\sigma, \delta) = \delta(s_{(0, j-1)}(\sigma, \delta)) $ and $s_{(i,0)}(\sigma, \delta) = \sigma(s_{(i-1,0)}(\sigma,\delta))$.
Therefore
\begin{align*}
\tau(x^{l} \ot a) &= \sum_{i + j = l} (\sigma(s_{(i-1,j)}(\sigma, \delta)) + \delta(s_{(i,j-1)}(\sigma, \delta)))(a) \ot x^{i}  \\
&= \sum_{i + j = l} s_{(i,j)}(\sigma, \delta)(a) \ot x^{i}.
\end{align*} 
Therefore equation $(2.9)$ holds and we see that  if $ s_{(i,j)}(\sigma, \delta) = 0$ for $i + j =n$ then $\tau$ induces a well defined multiplication on $A[\ox; \sigma, \delta]$.
\end{proof}
We end this section with some remarks on modules over twisted tensor products. 

\begin{defi}{\em
Let $A \ot_{\tau} B$ be a twisted tensor product algebra. A left $A$-module $M$ is  $\textit{compatible with} \, \, \,  \tau$ if there is a bijective $\K$-linear map $\tau_{B,M}: B \ot M \rightarrow M \ot B$ that commutes with the module structure of $M$ and multiplication in $B$. That is $\tau_{B,M}$ satisfies the relations 
\begin{equation}
\tau_{B,M} (m_{B} \ot 1) = (1 \ot m_{B})  (\tau_{B,M} \ot 1)  (1 \ot \tau_{B,M})
\end{equation}
\begin{equation}
\tau_{B,M}  (1 \ot \rho_{A,M}) = (\rho_{A,M} \ot 1) (1 \ot \tau_{B,M})  (\tau \ot 1)  
\end{equation}
where $\rho_{A,M}$ is the left $A$-module structure map.}
\end{defi}
 Note that a similar definition holds for a left $B$-module $N$ and the twisting map $\tau^{-1}$. 

If $M$ is a left $A$-module compatible with $\tau$ and $N$ is a left $B$-module then by \cite[Thm. 3.8]{CSV} we may give $M \ot N$ the structure of an $A \ot_{\tau} B $-left module via the composition of maps\\

\begin{tikzcd}[column sep=huge]
(A \ot_{\tau} B) \ot M \ot N \arrow[r, "1 \ot \tau_{B,M} \ot 1"] & A \ot M \ot B \ot N \arrow[r, "\rho_{A,M} \ot \rho_{B,N}"] & M \ot N .\\
\end{tikzcd}

The definition of compatibility with $\tau$ can also be extended to resolutions of modules as well. 
\begin{defi}{\em
 Let $M$ be a left $A$-module compatible with $\tau$ and $P_{\bu}(M)$ be a resolution of $M$. The resolution $P_{\bu}(M)$ is said to be $\textit{compatible with} \, \, \, \tau$ if there is a chain map $\tau_{B,\bu}:B \ot P_{\bu}(M) \rightarrow P_{\bu}(M) \ot B $ such that each $P_{i}(M)$ is compatible with $\tau$ via $\tau_{B,i}:B \ot P_{i}(M) \rightarrow P_{i}(M) \ot B$ and $\tau_{B,\bu}$  lifts $\tau_{B,M}$. }
\end{defi}
We note that this definition has an analog for $B$-module resolutions.  

\section{Truncated Ore Extensions}\label{sec:ext}

$\textbf{Left Modules over Truncated Ore Extensions}$. 

Given $M$, a left module over some truncated Ore extension $A[\overline{x}; \sigma, \delta]$, we wish to construct a projective resolution for $M$. To do this we will adapt methods from \cite{RTTP}. These methods first depend upon our ability to view $A[\ox; \sigma, \delta]$ as a twisted tensor product. Then we must show that, upon restriction to a left $A$-module, $M$ is compatible with the associated twisting map $\tau$. Finally using a resolution of $M$ as a left $A$-module we will construct a resolution of $M$ as a $A \ot_{\tau} B  \cong A[\ox; \sigma, \delta]$-module.

Let $A$ be an associative algebra and $B = \K[x]/(x^n)$ for some $n \in \N$. Let $\sigma \in \textrm{Aut}_{\K}(A)$ and $\delta$ be a left $\sigma$-derivation of $A$ satisfying the conditions of Theorem 2.7. Hence we may view $A[\ox; \sigma, \delta]$ as the twisted tensor product $A \ot_{\tau} B$ where $\tau$ is the twisting map induced by the Ore relation. 

To show that $M$ is compatible with $\tau$ we construct a bijective $\K$-linear map $\tau_{B,M}: B \ot M \rightarrow M \ot B$.  We define $M^{\sigma}$ to be the $\K$-vector space $M$ equipped with $A$-module action given by $a\cdot_{\sigma} m = \sigma(a) \cdot m$ for all $a \in A$ and $m \in M$. Now suppose that upon restriction to $A$, there is an $A$-module isomorphism 
\begin{align}
\phi: M \to M^{\sigma}.
\end{align}
Theorem 3.2 will show that under certain conditions similar to the ones imposed on $\sigma$ and $\delta$, $M$ will be compatible with $\tau$ via the $\K$-linear map defined by setting
$$ \tau_{B,M}(1 \ot m) = m \ot 1 $$
$$ \tau_{B,M}(\ox \ot m) = \phi(m) \ot \ox + \ox \cdot m \ot 1  \, \, \, \, \, \, \, \, \, \, \, \, \,\textrm{ for all} \, \, \, \, \, m \in M$$

and then iterating with respect to relation $(2.11)$ to define $\tau_{B,M}(\ox^{k}\ot m)$. 
\vspace{5mm}

\begin{theo}
If $ \phi$ and $\ox \cdot$ satisfy the relations 
\begin{align}
 s_{(i,j)}(\phi, \ox \cdot) = 0 
\end{align}
as maps from $M$ to $M$ for all $i + j = n$ with $ 1 \leq i \leq n-1 $ and $1 \leq j \leq n-1$ then $M$ is compatible with $\tau$ via $\tau_{B,M}$. That is, $\tau_{B,M}$ satisfies relations $(2.11)$ and $(2.12)$. Note that here we are identifying $M^{\sigma}$ with $M$ as vector spaces for the purposes of notation.
\end{theo}

\begin{proof}
Using the above definition for $\tau_{B,M}$, iterating with respect to $(2.11)$, and following an inductive proof similar to the one given in the proof of Theorem $2.7$ gives the following
$$\tau_{B,M}(\ox^{k} \ot m) = \sum_{i + j = k} s_{(i,j)}(\phi, \ox \cdot)(m) \ot \ox^{i} \, \, \, \, \, \,  \mathrm{for}  \, \, \, \mathrm{all} \, \, \, \, k \leq n  .$$
Thus $\tau_{B,M}$ satisfies relation $(2.11)$ if
\begin{align*}
0 = \tau_{B,M}(\ox^{n} \ot m) &= \tau_{B,M}(m_{B} \ot 1)(\ox \ot \ox^{n-1} \ot m) \\
&=  \sum_{i + j = n} s_{(i,j)}(\phi, \ox \cdot)(m) \ot \ox^{i}. \\
\end{align*}
But since $\ox^{n} = 0$, it follows that $ s_{(i,j)}(\phi, \ox \cdot)$ is identically $0$ when $j = n$. Also since $\ox^{n} =0$ we see that $s_{(n,0)}(\phi, \ox \cdot)(m) \ot \ox^{n} =0$. Finally by assumption we have that $s_{(i,j)}(\phi, \ox \cdot) = 0$ for $1 \leq i \leq n-1$ and $1\leq j \leq n-1$  and hence $\tau_{B,M}$ satisfies relation $(2.11)$. 

We now consider the diagram corresponding to relation $(2.12)$:
\begin{center}
\begin{tikzcd}
B \ot A \ot M \arrow[d, "\tau \ot 1"] \arrow[rr, " 1 \ot \rho_{A,M}"]& &   B \ot M \arrow[d, "\tau_{B,M}"] \\
 A \ot B \ot M \arrow[r, "1 \ot \tau_{B,M}"] &  A \ot M \ot B \arrow[r, " \rho_{A,M} \ot 1"] & M \ot B   \\
\end{tikzcd}
\end{center}
Since $\tau$, $\tau_{B,M}$, and $\rho_{A,M}$ are all $\K$-linear, in order to prove the diagram commutes it is sufficient to check that the composition of maps agree on elements of the form $\ox^{k} \ot a \ot m$ for all $k$, $0 \leq k \leq n-1$, and all $a \in A$, $m \in M$. For $k = 1$ we have 
\begin{align*}
&(\rho_{A,M} \ot 1)  (1 \ot \tau_{B,M})  (\tau \ot 1) (\ox \ot a \ot m ) \\
& =  (\rho_{A,M} \ot 1)  (1 \ot \tau_{B,M})(\sigma(a) \ot \ox \ot m + \delta(a) \ot 1 \ot m) \\
&=  (\rho_{A,M} \ot 1) (\sigma(a) \ot \phi(m) \ot \ox + \sigma(a) \ot \ox \cdot m \ot 1 + \delta(a) \ot m \ot 1) \\
&= \sigma(a) \cdot \phi(m) \ot \ox +( \sigma(a) \ox  + \delta(a) )\cdot m \ot 1  \\
&= \phi ( a \cdot m) \ot \ox + \ox a \cdot m \ot 1 \\
&= \tau_{B,M}(\ox \ot a \cdot m) = \tau_{B,M}(1 \ot \rho_{A,M} )( \ox \ot a \ot m).
\end{align*}
Now we assume that $k > 1$ and for all $l < k$ we have 
$$\tau_{B,M} (1 \ot \rho_{A,M}) (\ox^{l} \ot a \ot m )= (\rho_{A,M} \ot 1)  (1 \ot \tau_{B,M})  (\tau \ot 1) (\ox^{l} \ot a \ot m ).  $$

We consider the following diagram
\begin{center}
\begin{tikzcd}
 &  B \ot A \ot B \ot M  \arrow[r, " 1 \ot 1 \ot \tau_{B,M}"]&   B \ot A \ot M \ot B \arrow[dr, "1 \ot \rho_{A,M}  \ot 1"] &  \\
B \ot B \ot A \ot M \arrow[ur, "1 \ot \tau \ot 1"]  \arrow[d, " m_{B} \ot 1 \ot 1"] & &  & B \ot M \ot B \arrow[d, "\tau_{B,M} \ot 1"]\\
B \ot A \ot M \arrow[dr, "\tau \ot 1"] \arrow[rr, " 1 \ot \rho_{A,M}"]&   & B \ot M \arrow[dr, "\tau_{B,M}"]& M \ot B \ot B \arrow[d, "1 \ot m_{B}"]\\
 & A \ot B \ot M \arrow[r, "1 \ot \tau_{B,M}"] &  A \ot M \ot B \arrow[r, " \rho_{A,M} \ot 1"] & M \ot B   \\
\end{tikzcd}
\end{center}

Now since the map $m_{B}$ is surjective then for any  $\ox^{k} \ot a \ot m$ we have that $\ox^{k} \ot a \ot m \in \mathrm{im}(m_{B} \ot 1 \ot 1)$. In particular since $k > 1$ we may think of $\ox^{k} \ot a \ot m$ as the image of the element $\ox^{u} \ot \ox^{v} \ot a \ot m \in B \ot B \ot A \ot M$ for some $u + v = k$ with $u,v < k$. Thus given an element of the form $\ox^{u} \ot \ox^{v} \ot a \ot m \in B \ot B \ot A \ot M$, commutativity in the bottom portion of the diagram implies conditon $(2.12)$ for an element of the form $\ox^{k} \ot a \ot m$. We will first use a diagram chasing argument to show that the maps along the outside of the diagram take the same values on elements of the form $\ox^{u} \ot \ox^{v} \ot a \ot m$. We will do so by showing that the maps of some sub-diagrams take the same values on such elements. Consider the following diagram

\begin{center}
\begin{tikzcd}
 &  B \ot A \ot B \ot M \arrow[d, "\tau \ot 1 \ot 1"] \arrow[r, " 1 \ot 1 \ot \tau_{B,M}"]&   B \ot A \ot M \ot B \arrow[dr, "1 \ot \rho_{A,M}  \ot 1"] \arrow[d, "\tau \ot 1 \ot 1"]&  \\
B \ot B \ot A \ot M \arrow[ur, "1 \ot \tau \ot 1"]  \arrow[d, " m_{B} \ot 1 \ot 1"]& A \ot B \ot B \ot M \arrow[dd, "1 \ot m_{B} \ot 1"] \arrow[r, " 1 \ot 1 \ot \tau_{B,M}"]&   A \ot B \ot M \ot B  \arrow[d, "1 \ot \tau_{B,M} \ot 1"]  & B \ot M \ot B \arrow[d, "\tau_{B,M} \ot 1"]\\
B \ot A \ot M \arrow[dr, "\tau \ot 1"]&   &A \ot M \ot B \ot B \arrow[d, "1 \ot 1 \ot m_{B}"] & M \ot B \ot B \arrow[d, "1 \ot m_{B}"]\\
 & A \ot B \ot M \arrow[r, "1 \ot \tau_{B,M}"] &  A \ot M \ot B \arrow[r, " \rho_{A,M} \ot 1"] & M \ot B   \\
\end{tikzcd}
\end{center}
\vspace{3mm}
We see that in the following sub-diagram the $m$ in $\ox^{u} \ot \ox^{v} \ot a \ot m$ remains untouched. \\
\begin{center}
\begin{tikzcd}
 &  B \ot A \ot B \ot M \arrow[d, "\tau \ot 1 \ot 1"] \\
B \ot B \ot A \ot M \arrow[ur, "1 \ot \tau \ot 1"]  \arrow[d, " m_{B} \ot 1 \ot 1"]& A \ot B \ot B \ot M \arrow[dd, "1 \ot m_{B} \ot 1"] \\
B \ot A \ot M \arrow[dr, "\tau \ot 1"]&   \\
 & A \ot B \ot M
\end{tikzcd} 
\end{center} 
\vspace{3mm}
Hence we may show that the indicated composition of maps takes the same value on an element of the form $\ox^{u} \ot \ox^{v} \ot a \ot m$ by applying relation $(2.2)$ to an element of the form $\ox^{u} \ot \ox^{v} \ot a \ot 1$. We also have that the indicated composition of maps of the following diagram \vspace{3mm}
\begin{center}
\begin{tikzcd}
 B \ot A \ot B \ot M \arrow[d, "\tau \ot 1 \ot 1"] \arrow[r, " 1 \ot 1 \ot \tau_{B,M}"]&   B \ot A \ot M \ot B  \arrow[d, "\tau \ot 1 \ot 1"] \\
A \ot B \ot B \ot M \arrow[r, " 1 \ot 1 \ot \tau_{B,M}"] & A \ot B \ot M \ot B 
\end{tikzcd}
\end{center} \vspace{3mm}
take the same value on our element because the vertical and horizontal maps act on different factors. Hence regardless of the direction taken the same maps are applied to the same elements. The maps in the diagram 
\vspace{3mm}
\begin{center}
\begin{tikzcd}
 A \ot B \ot B \ot M \arrow[dd, "1 \ot m_{B} \ot 1"] \arrow[r, " 1 \ot 1 \ot \tau_{B,M}"]&   A \ot B \ot M \ot B  \arrow[d, "1 \ot \tau_{B,M} \ot 1"] \\
 &A \ot M \ot B \ot B \arrow[d, "1 \ot 1 \ot m_{B}"] \\
 A \ot B \ot M \arrow[r, "1 \ot \tau_{B,M}"] &  A \ot M \ot B
\end{tikzcd} 
\end{center} 
\vspace{3mm} 
take the same value on the element $\ox^{u} \ot \ox^{v} \ot a \ot m$ as a result of using the relation $(2.11)$ to define $\tau_{B,M}$. To show that the maps in the diagram
\vspace{3mm}
\begin{center}
\begin{tikzcd}
 B \ot A \ot M \ot B \arrow[dr, "1 \ot \rho_{A,M}  \ot 1"] \arrow[d, "\tau \ot 1 \ot 1"]&  \\
 A \ot B \ot M \ot B  \arrow[d, "1 \ot \tau_{B,M} \ot 1"]  & B \ot M \ot B \arrow[d, "\tau_{B,M} \ot 1"]\\
A \ot M \ot B \ot B \arrow[d, "1 \ot 1 \ot m_{B}"] & M \ot B \ot B \arrow[d, "1 \ot m_{B}"]\\
A \ot M \ot B \arrow[r, " \rho_{A,M} \ot 1"] & M \ot B   \\
\end{tikzcd} \end{center} \vspace{3mm}
take the same value we break it into two parts. We start with the following:
\vspace{3mm} \\
\begin{center}
\begin{tikzcd}
 B \ot A \ot M \ot B \arrow[dr, "1 \ot \rho_{A,M}  \ot 1"] \arrow[d, "\tau \ot 1 \ot 1"]&  \\
 A \ot B \ot M \ot B  \arrow[d, "1 \ot \tau_{B,M} \ot 1"]  & B \ot M \ot B \arrow[d, "\tau_{B,M} \ot 1"]\\
A \ot M \ot B \ot B \arrow[r, "\rho_{A,M} \ot 1 \ot 1"] & M \ot B \ot B \\
\end{tikzcd}
\end{center} 
\vspace{3mm}

Again assuming we started in $B \ot B \ot A \ot M$ with the element $\ox^{u} \ot \ox^{v}  \ot a \ot m$ and mapping through $B \ot A \ot B \ot M$  and into $B \ot A \ot M \ot B$ by the map $(1 \ot 1 \ot \tau_{B,M})(1 \ot \tau \ot 1)$ we see that the $\ox^{u}$ factor remains untouched. Thus the element in $B \ot A \ot M \ot B$ that we will be computing with will be a sum of elements of the form $\ox^{u} \ot a' \ot m' \ot b$ for some $a' \in A, m' \in M, b \in B$ . And since $\tau$ and $\tau_{B,M}$ are $\K$-linear it is enough to show that the compositions take the same values on $\ox^{u} \ot a' \ot m' \ot b$. But this is easily shown by a direct application of the induction hypothesis and the fact that $b$ remains untouched in the diagram. Finally we see that the composition of maps in 
\vspace{3mm}
\begin{center}
\begin{tikzcd}
A \ot M \ot B \ot B \arrow[r, "\rho_{A,M} \ot 1 \ot 1"]  \arrow[d, "1 \ot 1 \ot m_{B}"]   & M \ot B \ot B  \arrow[d, "1 \ot m_{B}"]\\
A \ot M \ot B \arrow[r, " \rho_{A,M} \ot 1"] & M \ot B   \\
\end{tikzcd}
\end{center}  
take the same value on our element because the vertical and horizontal maps act on separate factors. Hence regardless of the direction taken the same maps are applied to the same elements. Putting all these results together with the linearity of our maps and the fact that elements of the form $\ox^{u} \ot \ox^{v} \ot a \ot m$ with $0 \leq u \leq n-1$, $0\leq v \leq n-1$ form a vector space basis of $B \ot B \ot A \ot M$ we have established the fact that the following diagram commutes.

\begin{center}
\begin{tikzcd}
 &  B \ot A \ot B \ot M  \arrow[r, " 1 \ot 1 \ot \tau_{B,M}"]&   B \ot A \ot M \ot B \arrow[dr, "1 \ot \rho_{A,M}  \ot 1"] &  \\
B \ot B \ot A \ot M \arrow[ur, "1 \ot \tau \ot 1"]  \arrow[d, " m_{B} \ot 1 \ot 1"]&  &  & B \ot M \ot B \arrow[d, "\tau_{B,M} \ot 1"]\\
B \ot A \ot M \arrow[dr, "\tau \ot 1"]&   & & M \ot B \ot B \arrow[d, "1 \ot m_{B}"]\\
 & A \ot B \ot M \arrow[r, "1 \ot \tau_{B,M}"] &  A \ot M \ot B \arrow[r, " \rho_{A,M} \ot 1"] & M \ot B   \\
\end{tikzcd}
\end{center}
\vspace{3mm}

Now we consider the following diagram
\begin{center}
\begin{tikzcd}
 &  B \ot A \ot B \ot M  \arrow[r, " 1 \ot 1 \ot \tau_{B,M}"]&   B \ot A \ot M \ot B \arrow[dr, "1 \ot \rho_{A,M}  \ot 1"] &  \\
B \ot B \ot A \ot M \arrow[ur, "1 \ot \tau \ot 1"]  \arrow[d, " m_{B} \ot 1 \ot 1"] \arrow[rr, "1 \ot 1 \ot \rho_{A,M}"]& &  B \ot B \ot M \arrow[r, "1 \ot \tau_{B,M}"] \arrow[d, "m_{B} \ot 1"]& B \ot M \ot B \arrow[d, "\tau_{B,M} \ot 1"]\\
B \ot A \ot M \arrow[dr, "\tau \ot 1"] \arrow[rr, " 1 \ot \rho_{A,M}"]&   & B \ot M \arrow[dr, "\tau_{B,M}"]& M \ot B \ot B \arrow[d, "1 \ot m_{B}"]\\
 & A \ot B \ot M \arrow[r, "1 \ot \tau_{B,M}"] &  A \ot M \ot B \arrow[r, " \rho_{A,M} \ot 1"] & M \ot B   \\
\end{tikzcd}
\end{center}
\vspace{3mm}
If we again start with an element of the form $\ox^{u} \ot \ox^{v}  \ot a \ot m$ in $B \ot B \ot A \ot M$  then the maps of \vspace{3mm}
\begin{center}
\begin{tikzcd}
 &  B \ot A \ot B \ot M  \arrow[r, " 1 \ot 1 \ot \tau_{B,M}"]&   B \ot A \ot M \ot B \arrow[dr, "1 \ot \rho_{A,M}  \ot 1"] &  \\
B \ot B \ot A \ot M \arrow[ur, "1 \ot \tau \ot 1"]  \arrow[rr, "1 \ot 1 \ot \rho_{A,M}"]& &  B \ot B \ot M \arrow[r, "1 \ot \tau_{B,M}"]& B \ot M \ot B\\
\end{tikzcd}
\end{center}
\vspace{3mm}
give the same result because of the induction hypothesis and the fact that the $\ox^{u}$ factor goes untouched. The maps of the diagram 
\vspace{3mm}
\begin{center}
\begin{tikzcd}
B \ot B \ot A \ot M  \arrow[d, " m_{B} \ot 1 \ot 1"] \arrow[rr, "1 \ot 1 \ot \rho_{A,M}"]& &  B \ot B \ot M \arrow[d, "m_{B} \ot 1"]\\
B \ot A \ot M  \arrow[rr, " 1 \ot \rho_{A,M}"]&   & B \ot M \\
\end{tikzcd}
\end{center} \vspace{3mm}
clearly give the same result on our element. And finally the maps in  
\vspace{3mm}
\begin{center}
\begin{tikzcd}
 B \ot B \ot M \arrow[r, "1 \ot \tau_{B,M}"] \arrow[d, "m_{B} \ot 1"]& B \ot M \ot B \arrow[d, "\tau_{B,M} \ot 1"]\\
 B \ot M \arrow[dr, "\tau_{B,M}"]& M \ot B \ot B \arrow[d, "1 \ot m_{B}"]\\
 & M \ot B   \\
\end{tikzcd} 
\end{center} 
\vspace{3mm}
give the same result on our element because we used condition  $(2.11)$ to construct $\tau_{B,M}$. Now given that the maps on the outside of the diagram commute, the fact that compositions of maps of the previous three sub-diagrams give the same results on our element, and the surjectivity of $m_{B}$ we see that the following diagram commutes.
\vspace{3mm}
\begin{center}
\begin{tikzcd}
B \ot A \ot M \arrow[d, "\tau \ot 1"] \arrow[rr, " 1 \ot \rho_{A,M}"]&   &  B \ot M \arrow[d, "\tau_{B,M}"] \\
 A \ot B \ot M \arrow[r, "1 \ot \tau_{B,M}"] &  A \ot M \ot B \arrow[r, " \rho_{A,M} \ot 1"] & M \ot B   \\
\end{tikzcd}
\end{center}
\vspace{3mm}
Hence $\tau_{B,M}$ satisfies property $(2.12)$.
\end{proof}

Hence we now have conditions on $M$ which guarantee that it will be compatible with $\tau$. Namely from here out we will assume that $M$ is an $A[\ox; \sigma, \delta]$-module for which the $A$-module isomorphism 3.1, $\phi:M \rightarrow M^{\sigma}$, exists such that $s_{i,j}(\phi, \ox \cdot) = 0$ for all $i + j = n$ with $1 \leq i \leq n-1$ and $1 \leq j \leq n-1$.

Let $P_{\bu}(M)$ be a free resolution of $M$ as a left $A$-module
$$\begin{tikzcd}
P_{\bu}(M):  \cdots \arrow[r, "d_{2}"]& P_{1}(M) \arrow[r, "d_{1}"] & P_{0}(M)  \arrow[r, "\mu"] & M  \arrow[r] &  0.
 \end{tikzcd}$$ Our next step in the construction involves taking this resolution and showing that it is compatible with $\tau$. To do so we need a chain map $\tau_{B,\bu}:B \ot P_{\bu}(M) \rightarrow P_{\bu}(M) \ot B$. In particular we will use a chain map that takes inspiration from our twisting map $\tau$ and uses two other chain maps we will call $\sigma_{\bu}$ and $\delta_{\bu}$. We proceed by first constructing $\sigma_{\bu}$.

Using the above resolution $P_{\bu}(M)$, we construct another free resolution of $M$ 
 $$\begin{tikzcd}
P^{\sigma}_{\bu}(M):  \cdots \arrow[r, "d_{2}"]& P_{1}^{\sigma}(M) \arrow[r, "d_{1}"] & P_{0}^{\sigma}(M)  \arrow[r, "\phi^{-1}\mu"] & M  \arrow[r] &  0
 \end{tikzcd}$$
by using the module action $a \cdot_{\sigma} m = \sigma(a)\cdot m$ and setting $P_{i}^{\sigma}(M) = (P_{i}(M))^{\sigma}$  for each $i$. Then by the comparison theorem there exists an $A$-module chain map from $P_{\bu}(M)$ to $P^{\sigma}_{\bu}(M)$ which lifts the identity map on $M$. We may view this map as a $\K$-linear chain map

\begin{equation}
\sigma_{\bu}:P_{\bu}(M) \rightarrow P_{\bu}^{\sigma}(M)
\end{equation}
and note that $\sigma_{i}(a\cdot z) = \sigma(a)\cdot \sigma_{i}(z)$ for all $i \geq 0$, $a \in A$, and $z \in P_{i}(M)$. 

Before we construct our chain map  $\tau_{B,\bu}$, we must first define a left $A[\ox; \sigma, \delta]$-module action on the free $A$-modules $P_{\bu}(M)$. The following two lemmas mirror lemmas found in \cite{RTTP} and \cite{GOSR}. We show that the results still hold in the case of truncated Ore extensions. The first lemma gives the method for extending the $A$-module action to an $A[\ox; \sigma, \delta]$-action and the second gives the existence of the chain map $\delta_{\bu}$ that we need to define $\tau_{B,\bu}$. 

\begin{lemma}
Let $A$ be an associative algebra and $A[\ox; \sigma, \delta]$ be a truncated Ore extension. For any free $A$-module, $P$, there is an $A[\ox; \sigma, \delta]$-module structure on $P$ that extends the action of $A$.
\end{lemma}

\begin{proof}
We begin by first taking $P$ to be the free $A$-module $A$. As in \cite{RTTP} we define a left $A[x; \sigma, \delta]$-module action by letting $x$ act on $A$ by $x \cdot a= \delta(a)$ for all $a \in A$. Since $A[\ox; \sigma, \delta]$ is a truncated Ore extension we have that $\delta^{n}(a) = 0$ for all $a \in A$ thus $x^{n}\cdot a = \delta^{n}(a) = 0$. Hence the action factors through $A[x; \sigma, \delta]$ to the quotient $A[\ox; \sigma, \delta]$. Also we have
\begin{align*}
\ox a \cdot a' &= \ox \cdot (a \cdot a') =\ox \cdot (aa') = \delta(aa') \\ 
&= \delta(a)a' +\sigma(a)\delta(a') = \delta(a)\cdot a' + \sigma(a)( \ox \cdot a') \\
&= (\sigma(a)\ox + \delta(a)) \cdot a'
\end{align*}
for all $a, a' \in A$.

Now if $P$ is an arbitrary free $A$-module then $P \cong A^{\op I}$ for some index set $I$ and thus we let $\ox$ act on each summand in the manner shown above. We note as above that if we think about the action as coming from $A[x; \sigma, \delta]$ then $x^{n} \cdot z = 0$ since $x^{n}$ acts on any given $z \in P$ by acting with $x^{n}$ in each summand. Hence again the action factors through the quotient $A[\ox; \sigma, \delta]$. Also since $\ox$ acts in each summand it is trivial to show that $\ox a $ acts as $\sigma(a)\ox + \delta(a)$ on $P$ for all $a \in A$.  
 Hence every free $A$-module $P$ also has an $A[\ox; \sigma, \delta]$ structure which extends the action of $A$. 
\end{proof}

Let $M$ be an $A[\ox; \sigma, \delta]$-module as above and $P_{\bu}(M)$ be a free resolution of $M$ as an $A$-module. Let $f:M \rightarrow M$ be the function given by the action of $\ox$ on $M$, i.e. $f(m) = \ox \cdot m$. For our chain map $\tau_{B,\bu}$ we require a chain map $\delta_{\bu}$ which lifts $f$ and also plays nicely with the $A[\ox; \sigma, \delta]$-module action given in Lemma 3.5. The following lemma not only proves the existence of such a chain map but the body of the proof constitutes a method for constructing such a map. 

\begin{lemma}
There exists a $\K$-linear chain map $\delta_{\bu}: P_{\bu}(M) \rightarrow P_{\bu}(M)$ lifting $f:M \rightarrow M$ such that for each $j \geq 0$, $\delta_{j}(a\cdot z)= \sigma(a) \delta_{j}(z) + \delta(a)z $ for all $a \in A$ and $z \in P_{j}(M)$.
\end{lemma}

\begin{proof}
We let $P_{\bu}(M)$ be  the free resolution given by
$$\begin{tikzcd}
P_{\bu}(M):  \cdots \arrow[r, "d_{2}"]& P_{1}(M) \arrow[r, "d_{1}"] & P_{0}(M)  \arrow[r, "\mu"] & M  \arrow[r] &  0
 \end{tikzcd}$$
and $f:M \rightarrow M$ be defined as above. Let $j = 0$ and $\delta'_{0}$ be the map given by the action of $\ox$ on $P_{0}(M)$ as defined in Lemma $3.5$. That is  $\delta'_{0}(z) = \ox \cdot z$ for all $z \in P_{0}(M)$. If we again as in Lemma $3.5$ interpret the module actions as coming from $A[x; \sigma, \delta]$ and factoring through $A[\ox; \sigma, \delta]$, then a straightforward calculation shows that $\delta_{0}^{\prime n}(z) = 0$. 

Given $a \in A$ and $z \in P_{0}(M)$ we have $\delta'_{0}(az) = \ox \cdot az$. We identify the free $A$-module $P_{0}(M)$ with $A^{I}$ for some index set $I$. By Lemma $3.5$,  $\ox \cdot az$ is given by applying the action of $\ox$ on $A$ in each summand. Hence for each $i \in I$ we will have $\ox \cdot az_{i}$ where $z_{i} \in A$ is the $i^{\mathrm{th}}$ component of $z$. Since $\delta$ is a $\sigma$ derivation we have 
$$\ox \cdot az_{i} = \delta(az_{i})= \delta(a)z_{i} + \sigma(a)\delta(z_{i})$$
for each $i \in I$. Thus 
\begin{align*}
 \delta'_{0}(az) &= \ox \cdot az  \\
 &=\delta(a)z + \sigma(a)(\ox \cdot z) \\
 &= \delta(a)z + \sigma(a)\delta_{0}'(z)
\end{align*}

Now consider the map $\mu \delta'_{0} - f \mu: P_{0}(M) \rightarrow M^{\sigma}$. We may show that $\mu \delta'_{0} - f \mu$ is an $A$-module homomorphism via the calculations
\begin{align*}
(\mu \delta'_{0} - f \mu)(z + y) &= \mu \delta'_{0}(z +y) - f \mu(z +y) = \mu(\ox \cdot (z+y)) - f(\mu(z+y)) \\
&= \mu(\ox \cdot z) + \mu(\ox \cdot y) -  \ox \cdot (\mu(z) + \mu(y)) \\
&=( \mu(\ox \cdot z) - \ox \cdot \mu(z) ) + (\mu( \ox \cdot y)  - \ox \cdot \mu(y)) \\
&=(\mu \delta'_{0} - f \mu)(z) + (\mu \delta'_{0} - f \mu)(y)
\end{align*}
and
\begin{align*}
(\mu \delta'_{0} - f \mu)(az) &= \mu (\delta'_{0}(az) ) - f(\mu(az)) = \mu(\ox \cdot az) - \ox \cdot \mu(az) \\
&= \mu(\ox a \cdot z) - \ox a \cdot \mu(z)  = \mu( (\sigma(a) \ox + \delta(a) ) \cdot z ) - (\sigma(a) \ox + \delta(a)) \cdot \mu(z) \\
&= \mu(\sigma(a) \ox  \cdot z) + \delta(a) \cdot \mu(z) - \sigma(a) \ox \cdot \mu(z) - \delta(a) \cdot \mu(z) \\
&= \sigma(a) \cdot \mu( \ox  \cdot z)- \sigma(a)\cdot( \ox \cdot \mu(z) )= a \cdot_{\sigma} \mu(\ox \cdot z) - a \cdot_{\sigma} ( \ox \cdot \mu(z)) \\
&= a \cdot_{\sigma} (\mu(\delta'_{0}(z)) - a \cdot_{\sigma} f(\mu(z))  = a \cdot_{\sigma}(\mu\delta'_{0} - f \mu)(z)
\end{align*}
for all $a \in A$ and $z,y \in P_{0}(M)$. Since $P_{0}(M)$ is projective there exists an $A$-module homomorphism $\delta^{\prime \prime}_{0}:P_{0}(M) \rightarrow P^{\sigma}_{0}(M)$ such that $(\mu \delta'_{0} - f \mu) = \mu \delta^{\prime \prime}_{0}$. Set $\delta_{0} = \delta'_{0} - \delta^{\prime \prime}_{0}$. Then
\begin{align*}
\mu \delta_{0} &= \mu(\delta'_{0} - \delta^{\prime \prime}_{0}) = \mu \delta'_{0} - \mu \delta^{\prime \prime}_{0} \\
&= \mu \delta'_{0} - ( \mu \delta'_{0} - f \mu) = f \mu
\end{align*}
and thus $\delta_{0}$ lifts $f$. Since both $\delta'_{0}$ and $\delta^{\prime \prime}_{0}$ are $\K$-linear, $\delta_{0}$ is $\K$-linear by construction. Finally 
\begin{align*}
\delta_{0}(az) &= \delta'_{0}(az) - \delta^{\prime \prime}_{0}(az) = (\sigma(a) \delta'_{0}(z) + \delta(a)z) - a \cdot_{\sigma} \delta^{\prime \prime}_{0}(z) \\
&= \sigma(a) \cdot (\delta'_{0}(z) - \delta^{\prime \prime}_{0}(z)) + \delta(a) z \\
&= \sigma(a) \delta_{0}(z) + \delta(a) z
\end{align*}
for all $a \in A$, $z \in P_{0}(M)$. Now we let $j > 0$ and assume that for all $0 \leq l < j$ there exist $\K$-linear maps $\delta_{l}:P_{l}(M) \rightarrow P_{l}(M)$ such that $\delta_{l}(az) = \sigma(a)\delta_{l}(z) + \delta(a)z$ and $d_{l}\delta_{l} = \delta_{l-1}d_{l}$ for all $a \in A, \, z \in P_{l}(M)$. Like before we define $\delta'_{j}:P_{j}(M) \rightarrow P_{j}(M)$ to be the action of $\ox$ on $P_{j}(M)$ given by Lemma $3.5$. Again a straightforward calculation shows 
\begin{align*}
\delta'_{j}(az) &= \ox a \cdot z = (\sigma(a) \ox + \delta(a)) \cdot z \\
&= \sigma(a) \delta'_{j}(z) + \delta(a) z.
\end{align*}
for all $a \in A$, $z \in P_{j}(M)$. Consider the map $d_{j}\delta'_{j} - \delta_{j-1}d_{j} : P_{j}(M) \rightarrow P_{j-1}^{\sigma}(M)$. We first see that it is an $A$-module homomorphism by 
\begin{align*}
(d_{j}\delta'_{j} - \delta_{j-1}d_{j})(z + y) &= d_{j}\delta'_{j}(z +y)  - \delta_{j-1}d_{j}(z+y) = d_{j}(\ox \cdot (z + y)) - \delta_{j-1}(d_{j}(z) + d_{j}(y)) \\
&=d_{j}(\ox \cdot z ) + d_{j}(\ox \cdot y) - \delta_{j-1}(d_{j}(z)) - \delta_{j-1}(d_{j}(y)) \\
&= (d_{j}\delta'_{j} - \delta_{j-1}d_{j})(z) + (d_{j}\delta'_{j} - \delta_{j-1}d_{j})(y)
\end{align*}
and 
\begin{align*}
(d_{j}\delta'_{j} - \delta_{j-1}d_{j})(az) &= d_{j}(\ox a \cdot z) - \delta_{j-1}(d_{j}(az))\\ 
&= d_{j}((\sigma(a) \ox + \delta(a)) \cdot z) - \delta_{j-1}(a \cdot d_{j}(z))\\
&= d_{j}(\sigma(a) \ox \cdot z) + d_{j}(\delta(a) \cdot z) -\delta_{j-1}(a \cdot d_{j}(z)) \\
&= \sigma(a) \cdot  d_{j}(\ox \cdot z) + \delta(a) \cdot d_{j}(z) - (\sigma(a) \delta_{j-1}(d_{j}(z)) + \delta(a) d_{j}(z)) \\
&= \sigma(a) \cdot  d_{j}(\ox \cdot z) - \sigma(a) \cdot \delta_{j-1}(d_{j}(z)) \\
&= a \cdot_{\sigma} ( d_{j} \delta'_{j} - \delta_{j-1} d_{j})(z).
\end{align*}
for all $a \in A$, $y,z \in P_{j}(M)$. By the induction hypothesis we have that $\delta_{j-1}$ is a chain map and $(d_{j-1} \delta_{j-1}) d_{j} =(\delta_{j-2} d_{j-1})d_{j} =0$. Hence $ d_{j-1}( d_{j}\delta'_{j} - \delta_{j-1}d_{j}) = 0$ and $\textrm{Im}(d_{j}\delta'_{j} - \delta_{j-1}d_{j}) \subset \textrm{Ker}(d_{j-1})= \textrm{Im}(d_{j})$. Since $P_{j}(M)$ is projective there exists an $A$-module homomorphism $\delta^{\prime \prime}_{j}:P_{j}(M) \rightarrow P_{j}^{\sigma}(M)$ such that $d_{j} \delta'_{j} - \delta_{j-1}d_{j} = d_{j}\delta^{\prime \prime}_{j}$. Let $\delta_{j} = \delta'_{j} - \delta^{\prime \prime}_{j}$, then by construction $\delta_{j}$ is $\K$-linear and 
\begin{align*}
d_{j}\delta_{j} &= d_{j}( \delta'_{j} - \delta^{\prime \prime}) = d_{j} \delta'_{j} - d_{j}\delta^{\prime \prime}_{j} \\
&= d_{j} \delta'_{j} -(d_{j} \delta'_{j} - \delta_{j-1}d_{j}) \\
&= \delta_{j-1} d_{j}.
\end{align*}
Finally for all $a \in A$ and $z \in P_{j}(M)$, 
\begin{align*}
\delta_{j}(az) &= \delta'_{j}(az) - \delta^{\prime \prime}_{j}(az)  = \ox a \cdot z - \sigma(a) \cdot \delta^{\prime \prime}_{j}(z) \\
&= \sigma(a) \ox \cdot z + \delta(a) \cdot z - \sigma(a) \cdot \delta^{\prime \prime}_{j}(z) \\
&=\sigma(a) \cdot ( \delta'_{j}(z) - \delta^{\prime \prime}_{j}(z)) + \delta(a) \cdot z \\
&= \sigma(a) \delta_{j}(z) + \delta(a)z.
\end{align*}
\end{proof}

Now finally we are ready to construct our chain map $\tau_{B, \bu}$. Since our chain map will draw inspiration from the standard Ore relation we end up with restrictions on $\sigma_{\bu}$ and $\delta_{\bu}$ which mirror the restrictions that we encountered when dealing with $\tau$ and $\tau_{B,M}$. 

\begin{lemma}
Let $A[\ox; \sigma , \delta]$, $M$, $P_{\bu}(M)$, and $\tau_{B,M}$  be defined as above. We assume $M$ is compatible with $\tau$ via $\tau_{B,M}$. Let $\sigma_{\bu}$ be the chain map $(3.4)$ and $\delta_{\bu}$ be the chain map constructed in Lemma $3.6$. If $\sigma_{\bu}$ and $\delta_{\bu}$ satisfy the relations 
$$ s_{(k,j)}(\sigma_{\bu}, \delta_{\bu}) =0$$
for all $k + j = n$ with $0 \leq k \leq n-1$ and $1 \leq j  \leq n$ then the resolution $P_{\bu}(M)$ is compatible with the twisting map $\tau$.
\end{lemma}

\begin{proof}
We define a $\K$-linear map $\tau_{B,i}:B \ot P_{i}(M) \rightarrow P_{i}(M) \ot B$ by taking
$$\tau_{B,i}(\ox \ot z) = \sigma_{i}(z) \ot \ox + \delta_{i}(z) \ot 1$$
 for all $z \in P_{i}(M)$ where  we then extend the map with respect to relation $(2.11)$ to obtain
\begin{equation}
\tau_{B,i}(\ox^{l} \ot z) = \sum_{ k + j = l} s_{(k,j)}(\sigma_{i}, \delta_{i})(z) \ot \ox^{k}.
\end{equation}
Thus in a situation similar to Theorems $2.7$ and $3.2$, $\tau_{B,i}$ satisfies relation $(2.11)$ if $s_{(k,j)}(\sigma_{\bu}, \delta_{\bu}) = 0$ for all $k + j = n$ with $0 \leq k \leq n-1$ and $1 \leq j  \leq n$. All that remains is to show that $\tau_{B,i}$ satisfies relation $(2.12)$. Now for any $a \in A$ and $z \in P_{i}(M)$
$$
\tau_{B,i}(\ox \ot az) = \sigma_{i}(az) \ot \ox + \delta_{i}(az) \ot 1 = \sigma(a)\sigma_{i}(z) \ot \ox + \sigma(a)\delta_{i}(z) \ot 1 + \delta(a)z \ot 1
$$
by the properties of $\sigma_{\bu}$ and $\delta_{\bu}$. Then a straightforward calculation gives
\begin{align*}
(\rho_{A,i} \ot 1)&(1 \ot \tau_{B,i})(\tau \ot 1)(\ox \ot a \ot z) \\
&= (\rho_{A,i} \ot 1)(1 \ot \tau_{B,i})(\sigma(a) \ot \ox  \ot z + \delta(a) \ot 1 \ot z) \\
&= (\rho_{A,i} \ot 1)(\sigma(a) \ot \sigma_{i}(z) \ot \ox + \sigma(a) \ot \delta_{i}(z) \ot 1 + \delta(a) \ot z \ot 1) \\
&=\sigma(a) \sigma_{i}(z) \ot \ox + \sigma(a)\delta_{i}(z) \ot 1 + \delta(a)z \ot 1)
\end{align*} 
for all $a \in A$, $ z \in P_{i}(M)$. Assume that for all $t < l$ we have 
$$\tau_{B,i} (1\ot \rho_{A,i}) (\ox^{t} \ot a \ot z )= (\rho_{A,i} \ot 1)  (1 \ot \tau_{B,i})  (\tau \ot 1) (\ox^{t} \ot a \ot z ).  $$
Then by a diagram chasing argument similar to the one found in Lemma $3.2$ we may show that 
$$\tau_{B,i} (1\ot \rho_{A,i}) (\ox^{l} \ot a \ot z )= (\rho_{A,i} \ot 1)  (1 \ot \tau_{B,i})  (\tau \ot 1) (\ox^{l} \ot a \ot z )  $$
and thus by induction on $l$ we see that condition $(2.12)$ holds for all elements of the form $\ox^{l} \ot az$.
\end{proof}

Hence we have shown that given an $A[\ox; \sigma, \delta]$-module $M$ such that $M \cong M^{\sigma}$ and a free resolution $P_{\bu}(M)$ of $M$ as a left $A$-module we may construct maps $\tau_{B,M}, \tau_{B,\bu}$ such that $M$ and $P_{\bu}(M)$ are compatible with $\tau$. Before the proof of our final theorem we introduce one more definition.

\begin{defi}{\em
Let $A \ot_{\tau} B$ be a twisted tensor product of $\K$-algebras. Let $M$ be a left $A$-module and $N$ be a left $B$-module. Let $P_{\bu}(M)$ and $P_{\bu}(N)$ be projective $A$- and  $B$-module resolutions of $M$ and $N$ respectively. We denote the differentials of $P_{\bu}(M)$ by $d'_{i}$ and the differentials of $P_{\bu}(N)$ by $d''_{j}$. The $\textit{twisted product complex}$, $X_{\bu}$, of $P_{\bu}(M)$ and $P_{\bu}(N)$ is the complex 
$$\begin{tikzcd}
\cdots \arrow[r] & X_{2} \arrow[r] & X_{1} \arrow[r] & X_{0} \arrow[r] & M \ot N \arrow[r] & 0.
\end{tikzcd}$$
where $$X_{k} = \bigoplus_{i + j = k} P_{i}(M) \ot P_{j}(N)$$
with the differentials given by
$$d_{k} = \sum_{i+j = k}(d'_{i} \ot 1 + (-1)^{i} \ot d''_{j}).$$
}
\end{defi}

We now let $P_{\bu}(B)$ be the standard projective resolution of $\K$ as a module over $B= \K[x]/(x^{n})$ with $\epsilon_{B}$ the augmentation map that takes $\ox$ to $0$:
$$\begin{tikzcd}
\cdots \arrow[r, "x \cdot"] & \K[x]/(x^{n}) \arrow[r, "x^{n-1} \cdot"] & \K[x]/(x^{n}) \arrow[r, "x \cdot"] & \K[x]/(x^{n}) \arrow[r, "\epsilon_{B}"] & \K \arrow[r] & 0.
\end{tikzcd}$$

\begin{theo}
Let $A[\ox; \sigma, \delta]= A \ot_{\tau} B$ be a truncated Ore extension. Let $M$ be a left $A[\ox; \sigma, \delta]$-module compatible with $\tau$ via $\tau_{B,M}$ for which $M \cong M^{\sigma}$ as $A$-modules. Let $P_{\bu}(M)$ be a free resolution of $M$ as a left $A$-module. Let $\sigma_{\bu}$ be the chain map of $(3.4)$, $\delta_{\bu}$ be the chain map of Lemma $3.6$, and assume $P_{\bu}(M)$ is compatible with $\tau$ via $\tau_{B,\bu}$, the chain map of Lemma $3.7$. Suppose that $\sigma_{i}:P_{i}(M) \rightarrow P_{i}(M)$ is bijective for each $i \geq 0$. Then the twisted product complex of $P_{\bu}(M)$ and $P_{\bu}(B)$ gives a projective resolution of $M$ as a left $A[\ox; \sigma, \delta]$-module. 
\end{theo}

\begin{proof}
 Let $X_{\bu}$ be the twisted product complex of $P_{\bu}(M)$ and $P_{\bu}(B)$. By assumption, $M$ and $P_{\bu}(M)$ are compatible with $\tau$ and thus by \cite[Thm. 5.8]{RTTP} and \cite[Thm. 5.9]{RTTP} the twisted product complex $X_{\bu}$ is an exact complex of left $A \ot_{\tau} B = A[\ox; \sigma, \delta]$-modules.
It is clear that as $A[\ox; \sigma, \delta]$-modules  
$$
A[\ox; \sigma, \delta] \ot_{A} P_{i}(M) \cong (A \ot_{\tau} B) \ot_{A} P_{i}(M) \cong (B \ot _{\tau^{-1}} A) \ot_{A} P_{i}(M)\cong B \ot P_{i}(M).
$$
Since $\sigma_{i} $ is bijective then we have that as vector spaces $B \ot P_{i}(M) \cong P_{i}(M)  \ot B$ via the map $\tau_{B,i}$ whose inverse is given by 
$$
z \ot \ox \mapsto \ox \ot \sigma_{i}^{-1}(z) - 1 \ot \delta_{i}(\sigma_{i}^{-1}(z)).
$$

We now consider the following diagram\\
\vspace{3mm}
\begin{center}
\begin{tikzcd}
(A \ot_{\tau} B) \ot B \ot P_{i}(M)  \arrow[d, " 1 \ot m_{B} \ot 1"] \arrow[rr, "1 \ot 1 \ot \tau_{B,i}"]& &  (A \ot_{\tau} B) \ot P_{i}(M) \ot B \arrow[d, "1 \ot \tau_{B,i} \ot 1"]\\
A\ot B \ot P_{i}( M) \arrow[d, "\tau^{-1} \ot 1"] \arrow[r, " 1 \ot \tau_{B,i}"]& A \ot P_{i}(M) \ot B \arrow[dr, "\rho_{A,i} \ot 1"]   & A \ot P_{i}(M) \ot B \ot B  \arrow{l}[swap]{1 \ot 1 \ot m_{B}} \arrow[d, "\rho_{A,i} \ot m_{B}"]\\
B \ot A \ot P_{i}(M) \arrow[r, "1 \ot \rho_{A,i}"] & B \ot P_{i}(M) \arrow[r, "\tau_{B,i}"] & P_{i}(M) \ot B \\
\end{tikzcd}
\end{center}
\vspace{3mm}

The diagram\\
\vspace{3mm}
\begin{center}
\begin{tikzcd}
 A \ot P_{i}(M) \ot B \arrow[dr, "\rho_{A,i} \ot 1"]   & A \ot P_{i}(M) \ot B \ot B  \arrow{l}[swap]{1 \ot 1 \ot m_{B}} \arrow[d, "\rho_{A,i} \ot m_{B}"]\\
 & P_{i}(M) \ot B \\
\end{tikzcd}
\end{center}
\vspace{3mm}
commutes because the maps act on different factors. The diagram\\
\vspace{3mm}
\begin{center}
\begin{tikzcd}
(A \ot_{\tau} B) \ot B \ot P_{i}(M)  \arrow[d, " 1 \ot m_{B} \ot 1"] \arrow[rr, "1 \ot 1 \ot \tau_{B,i}"]& &  (A \ot_{\tau} B) \ot P_{i}(M) \ot B \arrow[d, "1 \ot \tau_{B,i} \ot 1"]\\
A\ot B \ot P_{i}( M)  \arrow[r, " 1 \ot \tau_{B,i}"]& A \ot P_{i}(M) \ot B   & A \ot P_{i}(M) \ot B \ot B  \arrow{l}[swap]{1 \ot 1 \ot m_{B}} \\
\end{tikzcd}
\end{center}
\vspace{3mm}
commutes because $P_{i}(M)$ is compatible with $\tau$ and thus $\tau_{B,i}$ satisfies relation $(2.11)$. The diagram
\vspace{3mm}
\begin{center}
\begin{tikzcd}
A\ot B \ot P_{i}( M) \arrow[r, " 1 \ot \tau_{B,i}"]& A \ot P_{i}(M) \ot B \arrow[dr, "\rho_{A,i} \ot 1"]   & \\
B \ot A \ot P_{i}(M) \arrow[u, "\tau \ot 1"]   \arrow[r, "1 \ot \rho_{A,i}"] & B \ot P_{i}(M) \arrow[r, "\tau_{B,i}"] & P_{i}(M) \ot B \\
\end{tikzcd}
\end{center}
\vspace{3mm}
commutes because $P_{i}(M)$ is compatible with $\tau$ and thus $\tau_{B,i}$ satisfies relation $(2.12)$. Now putting these together and noting that $(\tau^{-1} \ot 1)(\tau \ot 1)$ is the identity map we see that 
\vspace{3mm}
\begin{center}
\begin{tikzcd}
(A \ot_{\tau} B) \ot B \ot P_{i}(M)  \arrow[d, " 1 \ot m_{B} \ot 1"] \arrow[rr, "1 \ot 1 \ot \tau_{B,i}"]& &  (A \ot_{\tau} B) \ot P_{i}(M) \ot B \arrow[d, "1 \ot \tau_{B,i} \ot 1"]\\
A\ot B \ot P_{i}( M) \arrow[d, "\tau^{-1} \ot 1"]&  & A \ot P_{i}(M) \ot B \ot B   \arrow[d, "\rho_{A,i} \ot m_{B}"]\\
B \ot A \ot P_{i}(M) \arrow[r, "1 \ot \rho_{A,i}"] & B \ot P_{i}(M) \arrow[r, "\tau_{B,i}"] & P_{i}(M) \ot B \\
\end{tikzcd}
\end{center}
\vspace{3mm}
commutes. Hence $\tau_{B,i}$ preserves the module structure and is thus an $A[\ox; \sigma, \delta]$-module isomorphism. Thus we have that for every $i \geq 0$, $A[\ox; \sigma, \delta] \ot_{A} P_{i}(M) \cong P_{i}(M)  \ot B = Y_{i,j}$ as left $A[\ox; \sigma, \delta]$-modules. Since $ P_{i}(M)$ is a free $A$-module for each $i$ then we have that $P_{i}(M)  \cong A^{\op J}$ for some index set $J$. Thus 
$$A[\ox; \sigma, \delta] \ot_{A} P_{i}(M) \cong A[\ox; \sigma, \delta] \ot_{A} A^{\op n_{i}} \cong (A[\ox; \sigma, \delta] \ot_{A} A)^{\op n_{i}} \cong A[\ox; \sigma, \delta]^{\op n_{i}} $$
and we see that $A[\ox; \sigma, \delta] \ot_{A} P_{i}(M)$ is a free $A[\ox; \sigma, \delta]$-module. Therefore $A[\ox; \sigma, \delta] \ot_{A} P_{i}(M)$ is a free $A[\ox; \sigma, \delta]$-module and thus projective. 
\end{proof}

\section{Example}\label{sec:example}

For our example we will construct a resolution for a class of truncated Ore extensions which  includes the Nichols algebras that were used in \cite{NWW} to prove a finite generation of cohomology result.

Let $\mathbbm{k}$ be a field of prime characteristic $p$. We now consider the family of truncated Ore extensions of the form $A[\ox; \sigma, \delta] = A \ot_{\tau}B$ for $A =  \mathbbm{k}[x_{1}]/(x_{1}^{p})$, $B = \mathbbm{k}[x_{2}]/(x_{2}^{p})$, $\sigma$ the identity map on $A$, and $\delta$ the $\sigma$-derivation defined by 
\begin{align*}
\delta(1) = 0 \, \, \, \, \textrm{and} \, \, \, \, \delta(\oxa) = \alpha \oxa^{t}
\end{align*}
for $\alpha \in \mathbbm{k}$ and $2 \leq t \leq p- 1 $. 

As in section $2$ we wish to think of our twisting map as coming from the Ore extension $A[x_{2}; \sigma, \delta]$, thus we first define a twisting map $\tau$ for $A[x_{2}; \sigma, \delta] $ using the Ore relation:
$$\tau(\oxb \ot \oxa) =\sigma(\oxa) \ot \oxb + \delta(\oxa) \ot 1 = \oxa \ot \oxb + \delta(\oxa) \ot 1 $$ and then extend with respect to relation $(2.2)$ to obtain

\begin{align}
\tau(\oxb^{r} \ot \oxa^{s}) &= \sum_{j=0}^{r} {r \choose j} (s)^{[j]}\oxa^{s-j}\delta(\oxa)^{j} \ot \oxb^{r-j} \, \, \, \,\, \, \textrm{for} \,\,\, \textrm{all} \, \, r,s \in \N \\
&=  \sum_{j=0}^{r} {r \choose j} (s)^{[j]}\alpha^{j}\oxa^{s+j(t-1)} \ot \oxb^{r-j}
\end{align}

where $(s)^{[j]} $ is the generalized rising factorial 
$$(s)^{[j]}=\prod_{i=0}^{j-1}(s + i(t-1)), \, \, \, (s)^{[0]} = 1$$

\vspace{7mm}
Now that we have a formula for $\tau$ our next step is to show that it induces a well defined multiplication on the truncated Ore extension $A[\oxb; \sigma, \delta]$. Thus given $\sigma$ and $\delta$, we must show that they satisfy relation $2.8$ of Theorem $2.7$. Before we do so we will prove a useful fact about the special case when $\sigma = id_{A}$ and $\delta^{p} = 0$. 
\begin{lemma}
Let $\K$ be a field of characteristic $p$, $\Lambda$ be any associative $\K$-algebra, $\sigma = id_{\Lambda}$ be the identity automorphism of $\Lambda$, and $\delta:\Lambda \rightarrow \Lambda$ be any derivation for which $\delta^{p} = 0$. Then the standard twisting map $\tau$ of $\Lambda[x; \sigma, \delta] $ induces a well defined multiplication on $\Lambda[\ox; \sigma, \delta]$.
\end{lemma}

\begin{proof}
Since $\sigma = id_{A}$ it follows that $$s_{(i,j)}(\sigma, \delta) = {p \choose j}\delta^{j}.$$
For $j = 1, ..., p-1$, $p$ divides ${p \choose j}$ and since $\mathrm{char}(\K) = p$, $s_{(i,j)}(\sigma, \delta) = 0 $.  Thus by Theorem $2.7$, $\tau$ induces a well defined multiplication on $\Lambda[\ox; \sigma, \delta].$
\end{proof}
 
We next show that Lemma 4.3 applies to our family of truncated Ore extensions. 

\begin{lemma}
Let $\K$ be a field of characteristic $p$. Let $A = \K[x_{1}]/(x_{1}^{p})$ with $\sigma$, $\delta$ defined as above. Let $\tau:\K[x_{2}] \ot A \rightarrow A \ot \K[x_{2}]$ be the twisting map generated by the Ore relation on $\sigma$ and $\delta$. Then $\tau$ induces a well defined  multiplication on the truncated Ore extension $A[\oxb; \sigma, \delta] \cong A \ot_{\tau} B = \mathbbm{k}[x_{1}]/(x_{1}^{p}) \ot_{\tau} \mathbbm{k}[x_{2}]/(x_{2}^{p})$. 
\end{lemma}

\begin{proof}
By the previous lemma we need only show that $\delta^{p}=0$. A calculation shows that 
\begin{align*}
\delta^{p}(\oxa) = (\prod_{i =1}^{p}[(i-1)t -(i-2)])\alpha^{2}\oxa^{p(t-1)+1}
\end{align*}
and $t \geq 2$ implies that $p(t-1) +1 > p$ hence $\delta^{p} =0$.
\end{proof}

Thus from here on out we may think of the truncated Ore extension $A[\ox; \sigma, \delta]$ as a twisted tensor product with an associated twisting map $\tau$. Next we wish to construct a projective resolution of $\K$ as an $A \ot_{\tau} B$-module. To do so we follow the construction laid out in Section 3. First we restrict $\K$ to an $A$-module and show that it is compatible with $\tau$. Since $A = \mathbbm{k}[x_{1}]/(x_{1}^{p})$ and $B = \mathbbm{k}[x_{2}]/(x_{2}^{p})$, it follows that the $A[\ox; \sigma, \delta]$-module action on $\K$ is given by the augmentation map $\epsilon(a \ot b) = \epsilon_{A}(a)\epsilon_{B}(b)$ where $\epsilon_{A}:A \rightarrow \mathbbm{k}$ and $\epsilon_{B}:B \rightarrow \mathbbm{k}$ are the standard augmentation maps for $A$ and $B$ respectively. Also $\sigma = id_{A}$ implies that for any $z \in \K$, $\sigma(a)\cdot z = a \cdot z$ and thus $\K$ is trivially isomorphic to $\K^{\sigma}$. Following the construction from Section $3$ and noting that $\oxb$ acts as $0$ on $\K$ we define $\tau_{B,\K} :B \ot \K \rightarrow \K \ot B$ by
$$\tau_{B,\K}(1 \ot z) = z \ot 1$$
for all $z \in \K$. Clearly $\phi = id_{M}$ and $\oxb \cdot$ satisfy relation $(3.3)$ and thus by Lemma $3.2$, $\K$ is compatible with $\tau$ via the map $\tau_{B,\K}(b \ot z) = z \ot b $ for all $b \in B$, $z \in \K$. In particular we may note the following: 

\vspace{3mm}

\begin{lemma}
Let $\K$ be a field, $\Lambda$ be any associative $\K$-algebra, and $A[\ox,\sigma, \delta]$ be a truncated Ore extension of $\Lambda$ with $\sigma = id_{A}$. Let $M$ be a left $A[\ox,\sigma, \delta]$-module. If $\ox$ acts as $0$ on $M$ then $\tau_{B,M}$ as defined in Section $3$ is compatible with $\tau$.
\end{lemma}

\begin{proof}
The proof follows directly from Lemma $3.2$ and the fact that $\sigma = id_{A} $ implies that $\phi = id_{M}$.

\end{proof}

We now construct our chain map $\tau_{B, \bu}$ lifting $\tau_{B,\K}$ by first letting $P_{\bu}(A)$ be the standard   resolution of $\K$ as an $A$-module.
$$\begin{tikzcd}
P_{\bu}(A):  \cdots \arrow[r, "\oxa \cdot"]& A \arrow[r, "\oxa^{p -1} \cdot "] & A  \arrow[r, "\oxa \cdot"] & A  \arrow[r, "\epsilon_{A}"] & \mathbbm{k} \arrow[r] & 0
 \end{tikzcd}$$

Since $\sigma = id_{A}$ we may let
$$\sigma_{\bu}:P_{\bu}(A) \rightarrow P_{\bu}(A)$$
be given by $\sigma_{i} = id_{P_{i}(A)} $ for every $i$. Also since $P_{i}(A) = A$ for every $i$ we may set the $A[\oxb; \sigma, \delta]$-module action on $P_{i}(A)= A$ to be given by 
$$\oxb \cdot a = \delta(a)$$
for every $i$ and all $a \in A$.

\vspace{5mm}

We now have our $\sigma_{\bu}$ and an $A[\ox; \sigma, \delta]$-module action on our projective resolution. The next step in the construction of $\tau_{B,\bu}$ is to construct $\delta_{\bu}$. We do so by constructing  two maps for each  $i$ and then letting $\delta_{i}$ be the difference of those two maps. Hence for each $i$ we will have  $\delta_{i} = \delta_{i}' - \delta_{i}''$. We proceed with the construction of $\delta_{0}, \delta_{1}$, and $\delta_{2}$ given in the proof of Lemma $3.6$ and note that the construction of $\delta_{j}$ will be similar to $\delta_{1}$ if $j$ is odd and will be similar to $\delta_{2}$ if $j$ is even. 

Let $f:\K \rightarrow \K$ be the map given by the action of $\oxb$ on $\K$. Then $f(z) = 0$ for all $z \in \K$. Thus we have
\begin{align*}
\delta_{0}'(z) &= \oxb \cdot z = \delta(z)\\
(\epsilon_{A}\delta - f \epsilon_{A})(z)&=\epsilon_{A}(\delta(z)) - 0
\end{align*}
but since $\delta(z) \in (\oxb)$, $\epsilon_{A}(\delta(z)) = 0$. Hence $\delta_{0}'' =0$ and $\delta_{0} = \delta_{0}' - \delta_{0}'' = \delta$. Therefore 

$$\tau_{B,0}(\oxb \ot \oxa) = \sigma_{0}(\oxa) \ot \oxb + \delta_{0}(\oxa) \ot 1 =\oxa \ot \oxb + \delta(\oxa) \ot 1 = \tau(\oxb \ot \oxa)$$
Then extending by conditions $(2.11)$ and $(2.12)$ we obtain 
\begin{equation}
 \tau_{B,0}(\oxb^{r} \ot \oxa^{s}) =\sum_{j=0}^{r} {r \choose j} (s)^{[j]}(\alpha \oxa^{t})^{j}\oxa^{s-j} \ot \oxb^{r-j}.
\end{equation}
\vspace{5mm}

Starting on $\delta_{1}$ we set $\delta_{1}'(z) = \oxb \cdot z = \delta(z)$. Now let $g = \sum_{i=0}^{n} \beta_{i} \oxa^{i} \in A$ then
\begin{align*}
(d_{1} \delta_{1}' - \delta_{0} d_{1})(g) &= (\oxa \cdot \delta - \delta \oxa \cdot)(g) \\
&= \oxa \cdot \delta(g) - \delta(\oxa \cdot g) \\
&= \oxa\cdot (\sum_{i=0}^{n} \beta_{i} \delta(\oxa^{i})) - \delta(\sum_{i=0}^{n}\beta_{i}\oxa^{i+1}) \\
&= \oxa(\sum_{i=0}^{n}\beta_{i}[(i)\alpha \oxa ^{t +(i-1)})] - \sum_{i=0}^{n} \beta_{i}(i+1)\alpha \oxa^{t+ i} \\
&=\sum_{i=0}^{n}\beta_{i}(i) \alpha \oxa^{t+i} - \sum_{i=0}^{n} \beta_{i}(i+1) \alpha \oxa ^{t+i}\\
&=-(\sum_{i=0}^{n}\beta_{i}\alpha \oxa^{t+i}).
\end{align*} 
We need a map $\delta_{1}''$ such that $(d_{1} \delta_{1}' - \delta_{0} d_{1})(g) = \oxa \cdot \delta_{1}''(g)$ hence we define $\delta_{1}''$ on $\oxa^{i}$ by 
$$\delta_{1}''(\oxa^{i}) = -\alpha \oxa^{t+(i-1)}.$$
and then extend linearly to all elements of $A$.
Thus letting $\delta_{1} = \delta_{1}' - \delta_{1}''$ we have 
\begin{align*}
\delta_{1} (g) &= \delta_{1}'(g) - \delta_{1}''(g) \\
 &= \sum_{i=0}^{n}\beta_{i}(i)\alpha \oxa ^{t +(i-1)} - ( -\sum_{i=0}^{n}\beta_{i}\alpha \oxa^{t+(i-1)}) \\
&= \sum_{i=0}^{n}\beta_{i}(i+1)\alpha \oxa^{t + (i-1)}.
\end{align*}
Therefore 
\begin{align*}
\tau_{B,1}(\oxb \ot \oxa) &=   \sigma_{1}(\oxa) \ot \oxb + \delta_{1}(\oxa) \ot 1 \\
&= \oxa \ot \oxb + \delta(\oxa) \ot 1 - \delta_{1}''(\oxa) \ot 1 \\
&= \oxa \ot \oxb + \delta(\oxa) \ot 1 - (-\alpha \oxa^{t}) \ot 1 \\
&= \oxa \ot \oxb + 2\delta(\oxa) \ot 1.
\end{align*}
 We then extend the map using conditions $(2.11)$ and $(2.12)$ to obtain 
\begin{equation}
 \tau_{B,1}(\oxb^{r} \ot \oxa^{s})  \sum_{j = 0}^{r} {r \choose j} (s+1)^{[j]}(\alpha \oxa^{t})^{j}\oxa^{s -j} \ot \oxb^{r-j} .
\end{equation}
\vspace{5mm}

Then starting on $\delta_{2}$ we let $\delta_{2}'(z) = \delta(z)$. And since $im(\delta) \subset (\oxa)$ and $t\geq 2$ we have
\begin{align*}
d_{2}\delta_{2}' - \delta_{1} d_{2} &= \oxa^{p-1} \cdot \delta - (\delta_{1}' - \delta_{1}'')\oxa^{p-1}\cdot \\
&= \oxa^{p-1} \cdot \delta - \delta \oxa^{p-1}\cdot +\delta_{1}''\oxa^{p-1}\cdot \\
&= 0
\end{align*}
Hence we may choose $\delta_{2}'' = 0 $, $\delta_{2} = \delta$ and thus $\tau_{B,2} = \tau$. Finally we note that since the differentials of our projective resolution alternate between $\ox \cdot$ and $\ox^{p-1} \cdot$ then the chain maps $\tau_{B,i}$ themselves will also alternate. Hence these calculations of $\tau_{B,i}$ repeat for all remaining $i$ and we therefore give the following lemma. 
\vspace{5mm}

\begin{lemma}\label{comp}
For any integer $i \geq 0$, we define the chain map $\tau_{B,\bu}$ by letting $\tau_{B,i}: B \ot A \rightarrow A \ot B$ be the $\K$-linear map defined as follows:
$$\tau_{B,i}(\oxb^{r} \ot \oxa^{s}) =
\begin{cases} 
 \tau(\oxb^{r} \ot \oxa^{s}) = \sum_{j=0}^{r} {r \choose j} (s)^{[j]}(\alpha \oxa^{t})^{j}\oxa^{s-j} \ot \oxb^{r-j}    & i \, \,  is \, \, even \\
  \sum_{j = 0}^{r} {r \choose j} (s+1)^{[j]}(\alpha \oxa^{t})^{j}\oxa^{s -j} \ot \oxb^{r-j}     & i \,\,  is \, \, odd \\
 \end{cases}
$$

Then 

(a) Each $\tau_{B,i}$ is a bijective map whose inverse is 
$$\tau_{B,i}^{-1}(\oxa^{s} \ot \oxb^{r}) =
\begin{cases} 
 \sum_{j=0}^{r} {r \choose j} (s)^{[j]} \oxb^{r-j} \ot (-\alpha \oxa^{t})^{j} \oxa^{s-j} & i \, \,  is \, \, even \\
  \sum_{j = 0}^{r} {r \choose j} (s+1)^{[j]} \oxb^{r-j}\ot (-\alpha \oxa^{t})^{j} \oxa^{s-j}     & i \,\,  is \, \, odd \\
 \end{cases}
$$

(b) $\tau_{B,\bu}$ is compatible with $\tau$. That is for every $i$, $\tau_{B,i}$ satisfies the following relations
\begin{equation}
\tau_{B,i}  (m_{B} \ot 1) = (1 \ot m_{B})  (\tau_{B,i} \ot 1)  (1 \ot \tau_{B,i})
\end{equation}
\begin{equation}
\tau_{B,i}  (1 \ot \rho_{A,i}) = (\rho_{A,i} \ot 1)  (1 \ot \tau_{B,i})  (\tau \ot 1)  
\end{equation}

(c) $\tau_{B,\bu}:B \ot P_{\bu}(A) \rightarrow P_{\bu}(A) \ot B$ is a chain map. \\
And thus $P_{\bu}(A)$ is compatible with $\tau$ via $\tau_{B,\bu}$. 

\end{lemma}

\begin{proof}
Let $\tau_{B,i}$ be defined as above and $\beta = (t-1)$. We first make a useful calculation. For any $ j, k \in \N$
 $$(x)^{[j +k]} = \prod_{i = 0}^{j+k - 1}(x + i\beta) = x ( x + \beta) .... (x + (j+k -1) \beta) =[ \prod_{i = 0}^{j-1} (x + i \beta)][\prod_{i = j}^{j+k -1}(x+i \beta)]$$
$$=[ \prod_{i = 0}^{j-1} (x + i \beta)][\prod_{i =0}^{k -1}(x+j\beta + i \beta)] =(x)^{[j]}(x + j\beta)^{[k]}$$

\vspace{5mm}

Also we will use the following fact
 $${r \choose h} {h \choose j} = \frac{r!}{(r-h)!h!}\cdot \frac{h!}{(h-j)!j!} = \frac{r!}{(r-h)!(h-j)!j!}$$
$$ \frac{r!}{j!(r-j)!} \cdot \frac{(r-j)!}{(h-j)!(r-h)!} = {r \choose j} {r-j \choose h-j}$$
for all $r,h,j \in \N$.
\vspace{10mm}

(a) Now by construction $\tau_{B,i}$ is $\K$-linear. We will show $\tau_{B,i}  \tau_{B,i}^{-1} = 1_{A \ot  B}$ for the case when $i$ is even. The proof of the remaining case is similar. 
\begin{align*}
&\tau  \tau^{-1}(\oxa^{s} \ot \oxb^{r}) = \tau(\sum_{j = 0}^{r} {r \choose j} (-\alpha)^{j}(s)^{[j]} \oxb^{r-j} \ot \oxa^{s + j \alpha}) \\
&= \sum_{j=0}^{r} {r \choose j} (-\alpha)^{j} (s)^{[j]} ( \sum_{k =0}^{r-j} {r-j \choose k} ( s + j \beta))^{[k]} \alpha^{k} \oxa^{kt} \oxa^{s +j(t-1) - k} \ot \oxb^{r-(j+k)}) \\
&=  \sum_{j=0}^{r}\sum_{k =0}^{r-j}{r \choose j}{r-j \choose k}(-1)^{j} \alpha^{j+k} (s)^{[j]}  ( s + j\beta)^{[k]} \oxa^{s + (j+k)t - (j+k)} \ot \oxb^{r- (j+k)} \\
&= \sum_{h = 0}^{r}(\sum_{j + k = h} {r \choose j}{r-j \choose k}(-1)^{j} \alpha^{h} (s)^{[h]} \oxa^{s + h\beta} \ot \oxb^{r - h})
\end{align*}

Thus to show that  $\tau_{B,i}  \tau_{B,i}^{-1} = 1_{A \ot  B}$ it suffices to show that 
$$\sum_{j + k = h} {r \choose j}{r-j \choose k}(-1)^{j}\alpha^{h} (s)^{[h]} = 
\begin{cases}
1 & h = 0 \\
0 & h \neq 0 \\
\end{cases}
$$
Now if $h = 0$ then both $j$ and $k$ must be $0$ and hence the sum is clearly equal to $1$. Now suppose $h \neq 0$. Then we have 
\begin{align*}
&\sum_{j + k = h} {r \choose j}{r-j \choose k}(-1)^{j} \alpha^{h} (s)^{[h]} = \alpha^{h} (s)^{[h]} ( \sum_{j = 0}^{h}(-1)^{j} {r \choose j}{r-j \choose h-j}) \\
&=  \alpha^{h} (s)^{[h]} ( \sum_{j = 0}^{h}(-1)^{j} {r \choose h} { h \choose j}) =   \alpha^{h} (s)^{[h]}{ r \choose h} ( \sum_{j = 0}^{h}(-1)^{j} {h \choose j}) = 0
\end{align*}
and hence $\tau_{B,i}  \tau_{B,i}^{-1} = 1_{A \ot  B}$. 
\vspace{7mm}

(b) We show the case for $i$ odd. The remaining case is similar. For relation $(4.9)$ we have
\begin{align*} 
&\tau_{B,i}  (m_{B} \ot 1)(\oxb^{r_{1}} \ot \oxb^{r_{2}} \ot \oxa^{s}) = \tau_{B,i}( \oxb^{r_{1} + r_{2}} \ot \oxa^{s}) \\
&= \sum_{j = 0}^{r_{1} + r_{2}} { r_{1} + r_{2} \choose j}(s + 1)^{[j]} \oxa^{s -j} (\alpha \oxa^{t})^{j} \ot \oxb^{r_{1} + r_{2} - j}
\end{align*}

and

\begin{align*}
&(1 \ot m_{B})  (\tau_{B,i} \ot 1)  (1 \ot \tau_{B,i})(\oxb^{r_{1}} \ot \oxb^{r_{2}} \ot \oxa^{s}) \\
& =  (1 \ot m_{B})  (\tau_{B,i} \ot 1) ( \oxb^{r_{1}} \ot [\sum_{j = 0}^{r_{2}}{ r_{2} \choose j } (s + 1)^{[j]}\oxa^{s - j} (\alpha \oxa^{t})^{j} \ot \oxb^{r_{2} -j}]    )   \\
&= (1 \ot m_{B})  (\tau_{B,i} \ot 1 )  ( \sum_{j = 0}^{r_{2}}{ r_{2} \choose j } (s + 1)^{[j]} \alpha^{j} \oxb^{r_1} \ot \oxa^{s + j\beta} \ot \oxb^{r_{2} -j}) \\
&=  (1 \ot m_{B})( \sum_{j = 0}^{r_{2}}{ r_{2} \choose j } (s + 1)^{[j]} \alpha^{j} [ \sum_{k = 0}^{r_{1}} {r_{1} \choose k}(s + j\beta+1)^{[k]} \oxa^{s + j\beta - k} \delta(\oxa)^{k} \ot \oxb^{r_{1} - k} ] \ot \oxb^{r_{2} - j})\\
&= \sum_{j = 0}^{r_{2}}{ r_{2} \choose j } (s + 1)^{[j]} \alpha^{j} ( \sum_{k = 0}^{r_{1}} {r_{1} \choose k}(s + j\beta +1)^{[k]} \oxa^{s + j(t-1) - k} (\alpha \oxa^{t})^{k} \ot \oxb^{r_{1} + r_{2} - (j + k)} ) \\
&= \sum_{j = 0}^{r_{2}}{ r_{2} \choose j } (s + 1)^{[j]}  ( \sum_{k = 0}^{r_{1}} {r_{1} \choose k}(s + j\beta +1)^{[k]} \oxa^{s - (j + k)}\alpha^{j+k}\oxa^{(j+k)t}  \ot \oxb^{r_{1} + r_{2} - (j + k)} ) \\
&=\sum_{h = 0}^{r_{1} + r_{2}} \sum_{j+k = h} { r_{2} \choose j }{r_{1} \choose k}  (s + 1)^{[j]}  (s + j\beta +1)^{[k]}  \oxa^{s-h} (\alpha \oxa^{t})^{h} \ot \oxb^{r_{1} + r_{2} -h} .  
\end{align*}

Thus to show that $\tau_{B,i}  (m_{B} \ot 1) = (1 \ot m_{B})  (\tau_{B,i} \ot 1)  (1 \ot \tau_{B,i}) $ it is sufficient to show 
$${ r_{1} + r_{2} \choose h}(s + 1)^{[h]}  = \sum_{j+k =h} { r_{2} \choose j }{r_{1} \choose k}  (s + 1)^{[j]}  (s + j\beta +1)^{[k]}.$$
Our previous calculation gives us $(s + 1)^{[h]} = (s + 1)^{[j]}  (s + j\beta +1)^{[k]}$. Finally  by a simple re-indexing and use of a well known identity we have 
$$\sum_{j+k =h} {r_{2} \choose j} {r_{1} \choose k} = \sum_{j =0}^{h} {r_{2} \choose j} { r_{1} \choose h-j} = {r_{1} + r_{2} \choose h}.$$
Therefore $\tau_{B,i}  (m_{B} \ot 1) = (1 \ot m_{B})  (\tau_{B,i} \ot 1)  (1 \ot \tau_{B,i}) $ for $i$ odd.

\vspace{10mm}

For relation $(4.10)$ we have
 \begin{align*}
\tau_{B,i}(1 \ot \rho_{A,i})(\oxb^{r} \ot \oxa^{s_{1}} \ot \oxa^{s_{2}}) &= \tau_{B, i}(\oxb^{r} \ot \oxa^{s_{1} + s_{2}})\\
&= \sum_{j = 0}^{r} {r \choose j} (s_{1} + s_{2} + 1)^{[j]} \oxa^{s_{1} + s_{2} - j}(\alpha \oxa^{t})^{j} \ot \oxb^{r - j}
\end{align*}
and 
\begin{align*}
& (\rho_{A,i} \ot 1) (1 \ot \tau_{B,i}) (\tau \ot 1) (\oxb^{r} \ot \oxa^{s_{1}} \ot \oxa^{s_{2}})  \\
&=  (\rho_{A,i} \ot 1) (1 \ot \tau_{B,i}) ([\sum_{j = 0}^{r}{r \choose j}(s_{1})^{[j]}\oxa^{s_{1} - j}(\alpha \oxa^{t})^{j} \ot \oxb^{r-j}] \ot \oxa^{s_{2}}) \\
&=  (\rho_{A,i} \ot 1)( \sum_{j = 0}^{r} {r \choose j} (s_{1})^{[j]}\oxa^{s_{1} -j}(\alpha \oxa^{t})^{j} \ot [\sum_{k =0}^{r-j} { r - j \choose k}(s_{2} +1)^{[k]} \oxa^{s_{2} -k}(\alpha \oxa^{t})^{k} \ot \oxb ^{r - j -k}] ) \\
&=\sum_{j = 0}^{r} \sum_{k= 0}^{r -j} { r \choose j} {r - j \choose k}(s_{1})^{[j]}(s_{2} +1 )^{[k]} \oxa^{s_{1} + s_{2} - (j+k)}(\alpha \oxa^{t})^{j+k} \ot \oxb^{r- (j+k)} \\
&= \sum_{h = 0}^{r}(\sum_{j + k = h} { r \choose j} {r - j \choose k}(s_{1})^{[j]}(s_{2} +1 )^{[k]}) \oxa^{s_{1} + s_{2} - h} (\alpha \oxa^{t})^{h} \ot \oxb^{r - h} .
\end{align*}

\vspace{7mm}

Now since ${r \choose h} {h \choose j}= {r \choose j} {r-j \choose h-j}$, we have that 
\begin{align*} 
\sum_{j + k = h} {r \choose j} {r - j \choose k}(s_{1})^{[j]}(s_{2} + 1)^{[k]} &=  \sum_{j = 0}^{h} {r \choose j} {r - j \choose h - j}(s_{1})^{[j]}(s_{2} + 1)^{[h-j]} \\
&=  \sum_{j = 0}^{h}  {r \choose h} {h \choose j}  (s_{1})^{[j]}(s_{2} + 1)^{[h - j]}.   
\end{align*}

Hence to show that $\tau_{B,i}  (1 \ot \rho_{A,i}) = (\rho_{A,i} \ot 1)  (1 \ot \tau_{B,i}) (\tau \ot 1)  $ we simply need to show that 
$$ (s_{1} + s_{2} + 1)^{[h]} = \sum_{j = 0}^{h}  {h \choose j}  (s_{1})^{[j]}(s_{2} + 1)^{[h-j]} .  $$

Let $x, y \in \N$. In the following we will use the previously shown fact that $(x)^{[j +1]} = x (x + \beta)^{[j]}$. Now we proceed by induction on $n$ to show that 
$$(x + y)^{[n]} = \sum_{j=0}^{n}{n \choose j} (x)^{[j]}(y)^{[n - j]}.$$
The case $n = 1$ is given by a straightforward calculation:

$$(x + y)^{[1]} = x + y = 1(y)^{[1]} + (x)^{[1]}1 = \sum_{j =0}^{1}{1 \choose j}(x)^{[j]}(y)^{[1-j]}.  $$
Assume that for $n = h-1$ we have $(x + y)^{[h-1]} = \sum_{j = 0}^{h-1} { h-1 \choose j}(x)^{[j]}(y)^{[h-1 -j]}$\\
Then for $n = h$ we have
\begin{align*}
&\sum_{j = 0}^{h} { h \choose j}(x)^{[j]}(y)^{ [ h -j ]} 
= \sum_{j = 0}^{h-1}{ h \choose j}(x)^{[j]}(y)^{[h -j]} + (x)^{[h]} \\
&= \sum_{j=0}^{h-1}({h-1 \choose j-1} + { h-1 \choose j})(x)^{[j]}(y)^{[h -j]} + (x)^{[h]} \\
&=  \sum_{j=0}^{h-1}{h-1 \choose j-1}(x)^{[j]}(y)^{[h -j]} + \sum_{j=0}^{h-1}{ h-1 \choose j}(x)^{[j]}(y)^{[h -j]} + (x)^{[h]}  \\
&=  \sum_{j=1}^{h-1}{h-1 \choose j-1}(x)^{[j]}(y)^{[h -j]} + \sum_{j=0}^{h-1}{ h-1 \choose j}(x)^{[j]}(y)^{[h -j]} + (x)^{[h]}   \\
&= \sum_{j = 0}^{h-2}{ h-1 \choose j} (x)^{[j+1]}(y)^{[h-1-j]} +  \sum_{j=0}^{h-1}{ h-1 \choose j}(x)^{[j]}(y)^{[h -j]} + (x)^{[h]} \\
&=  \sum_{j = 0}^{h-1}{ h-1 \choose j} (x)^{[j+1]}(y)^{[h-1-j]} +  \sum_{j=0}^{h-1}{ h-1 \choose j}(x)^{[j]}(y)^{[h -j]} \\
&= x  \sum_{j = 0}^{h-1}{ h-1 \choose j} (x + \beta)^{[j]}(y)^{[h-1-j]}  + y \sum_{j=0}^{h-1}{ h-1 \choose j}(x)^{[j]}(y + \beta)^{[h-1 -j]}
\end{align*}
By the induction hypothesis, 
\begin{align*}
\sum_{j = 0}^{h} { h \choose j}(x)^{[j]}(y)^{[h -j]} &= x ( x+ \beta + y)^{[h-1]} + y ( x + y + \beta)^{[h-1]} \\
&= (x + y)(x + y + \beta)^{[h-1]} = (x + y )^{[h]} 
\end{align*}
and therefore $(s_{1} + s_{2} + 1)^{[h]} = \sum_{j = 0}^{h}  {h \choose j}  (s_{1})^{[j]}(s_{2} + 1)^{[h-j]}   $. Hence  $$\tau_{B,i} (1 \ot \rho_{A,i}) = (\rho_{A,i} \ot 1)  (1 \ot \tau_{B,i})  (\tau \ot 1)  $$ for $i$ odd. 
\vspace{10mm}

(c) Consider the following diagram where $i$ is odd
$$\begin{tikzcd}[row sep=large, column sep = large]\label{dig}
 B \ot A \arrow[r, "1 \ot \oxa^{p -1} \cdot "] \arrow[d, dashrightarrow, "\tau"] & B \ot A  \arrow[r, "1 \ot \oxa \cdot"] \arrow[d, dashrightarrow, "\tau_{B,i}"] & B \ot A   \arrow[d, dashrightarrow, "\tau"] \\
 A \ot B \arrow[r, "\oxa^{p -1} \cdot \ot 1 "] & A \ot B  \arrow[r, "\oxa \cdot \ot 1"] & A \ot B  \\
\end{tikzcd}$$

Evaluating the right square of the diagram gives us
\begin{align*}
\tau (1 \ot \oxa)(\oxb^{r} \ot \oxa^{s}) &= \tau(\oxb^{r} \ot \oxa^{s+1}) \\
&= \sum_{j=0}^{r} {r \choose j} (s+1)^{[j]}\oxa^{s+1-j}(\alpha \oxa^{t})^{j} \ot \oxb^{r-j} \\
&= (\oxa \ot 1)(\sum_{j=0}^{r} {r \choose j} (s+1)^{[j]}\oxa^{s-j}(\alpha \oxa^{t})^{j} \ot \oxb^{r-j} ) \\
&= (\oxa \ot 1) \tau_{B,1}(\oxb^{r} \ot \oxa^{s}) .
\end{align*}

Evaluating the left square of the diagram gives us
\begin{align*}
\tau_{B,i}(1 \ot \oxa^{p -1} ) (\oxb^{r} \ot \oxa^{s}) &= \tau_{B,i}(\oxb^{r} \ot \oxa^{s + p -1}) \\
&=  \sum_{j=0}^{r} {r \choose j} (s+p)^{[j]}\oxa^{s+p-1-j}(\alpha \oxa^{t})^{j} \ot \oxb^{r-j}.
\end{align*}

Thus we have 
\begin{align*}
\tau_{B,i}(1 \ot \oxa^{p -1} ) (\oxb^{r} \ot \oxa^{s}) &=
\begin{cases}
0 \, \, \, \, \, \, \, \, \, \, \, \, \, for \, \,  s \neq 0 \\
 \sum_{j=0}^{r} {r \choose j} (p)^{[j]}\oxa^{p-1-j}(\alpha \oxa^{t})^{j} \ot \oxb^{r-j} \\
\end{cases} \\
&= 
\begin{cases}
0\, \, \, \, \, \, \, \, \, \, \, \, \, \, \, \, \, \, \, \, \, \, \, \, \, \, \, \, \, \, \, \, s > 0\\
\oxa^{p-1} \ot \oxb^{r} \, \, \, \, \, \, s = 0\\
\end{cases}
\end{align*}
and 
\begin{align*}
(\oxa^{p-1} \ot 1)\tau(\oxb^{r} \ot \oxa^{s}) &= (\oxa^{p-1} \ot 1) (\sum_{j=0}^{r} {r \choose j} (s)^{[j]}\oxa^{s-j}(\alpha \oxa^{t})^{j} \ot \oxb^{r-j}) \\
&= \sum_{j=0}^{r} {r \choose j} (s)^{[j]}\oxa^{s+(p-1)-j}(\alpha \oxa^{t})^{j} \ot \oxb^{r-j} \\
&=
\begin{cases}
0\, \, \, \, \, \, \, \, \, \, \, \, \, \, \, \, \, \, \, \, \, \, \, \, \, \, \, \, \, \, \, \, s > 0\\
\oxa^{p-1} \ot \oxb^{r} \, \, \, \, \, \, s = 0\\
\end{cases}.
\end{align*}
Therefore 
$$\begin{tikzcd}[row sep=large, column sep = large]\label{dig}
 B \ot A \arrow[r, "1 \ot \oxa^{p -1} \cdot "] \arrow[d, dashrightarrow, "\tau"] & B \ot A  \arrow[r, "1 \ot \oxa \cdot"] \arrow[d, dashrightarrow, "\tau_{B,i}"] & B \ot A   \arrow[d, dashrightarrow, "\tau"] \\
 A \ot B \arrow[r, "\oxa^{p -1} \cdot \ot 1 "] & A \ot B  \arrow[r, "\oxa \cdot \ot 1"] & A \ot B  \\
\end{tikzcd}$$
commutes. Hence by repeated application of this calculation we see that $\tau_{B,\bu}$ is a chain map. 
\end{proof}
\vspace{5mm}

Thus by Theorem $3.10$ if we are given the standard projective resolution of $\K$ as a left $A = \K[x_{1}]/(x_{1}^{p})$-module
$$\begin{tikzcd}
P_{\bu}(A):  \cdots \arrow[r, "\oxa^{p-1} \cdot"]& A \arrow[r, "\oxa \cdot"] & A  \arrow[r, "\epsilon_{A}"] & \K  \arrow[r] &  0
 \end{tikzcd}$$
and the standard projective resolution of $\K$ as a  left $B =\K[x_{2}]/(x_{2}^{p})$-module
$$\begin{tikzcd}
P_{\bu}(B):  \cdots \arrow[r, "\oxb^{p-1} \cdot"]& B \arrow[r, "\oxb \cdot"] & B \arrow[r, "\epsilon_{B}"] & \K  \arrow[r] &  0
 \end{tikzcd}$$
we may construct a projective resolution of $\K$ using the twisted product complex of the two resolutions. That is our projective resolution of $\K$ as a left $A \ot_{\tau} B$-module is given by:
$$\begin{tikzcd}
Y_{\bu}(\K):  \cdots \arrow[r, "d_{2}"]& Y_{1}\arrow[r, "d_{1}"] & Y_{0} \arrow[r, "\epsilon_{A}"] & \K  \arrow[r] &  0
 \end{tikzcd}$$
where 
$$
Y_{n} = \oplus_{i + j = n}Y_{i,j}  \, \, \, \,   \, \, \, \, \, \, \, \, \textrm{for} \, \, \, \, \, \, \, \, \, \, Y_{i,j} = P_{i}(A) \ot P_{j}(B) = A \ot B $$
and
$$
d_{ n} = \sum_{i + j = n} d_{i,j} \, \, \, \,   \, \, \, \, \, \, \, \, \textrm{for} \, \, \, \, \, \, \, \, \, \, d_{i,j} = d_{i,j}^{h} + d_{i,j}^{v}
$$
with $d_{i,j}^{h} = \oxa \cdot \ot 1$ for $i$ odd, $d_{i,j}^{h} = \oxa^{p-1} \cdot \ot 1$ for $i$ even, $d_{i,j}^{v} = (-1)^{i} \ot \oxb \cdot $ for $j$ odd, and finally $d_{i,j}^{v} = (-1)^{i} \ot \oxb^{p-1} \cdot$ for $j$ even.
Doing so gives the following projective resolution;

$$\begin{tikzcd}
\cdots \arrow[r, "d_{3}"] & (A \ot B)^{\op 3} \arrow[r, "d_{2}"] & (A \ot B)^{\op 2} \arrow[r, "d_{1}"] & A \ot B \arrow[r] & \K \arrow[r] & 0 .
\end{tikzcd}$$

\end{document}